\documentclass[final]{siamltex}
\usepackage{amsmath,braket,amsfonts}
\usepackage[caption=false]{subfig}
\usepackage{graphicx,epstopdf}

\usepackage{amssymb}
\usepackage{mathrsfs}

\makeatletter
\newcommand{\proofname}{Proof} 

\renewenvironment{proof}
  {\begin{trivlist}\item[]{\em \proofname.}}  
  {\hspace*{\fill}$\Box$\end{trivlist}}
\makeatother

\newtheorem{remark}{Remark}[section]

\newcommand{\stl}
{\mathrel{\raise2pt\hbox{${\mathop<\limits_{\raise1pt\hbox{\mbox{$\sim$}}}}$}}}

\newcommand{\stlm}
{\mathrel{\raise2pt\hbox{${\mathop<\limits_{\raise1pt\hbox{\mbox{$\sim$}}}}_{\raise4.5pt\hbox{\begin{footnotesize}$\mu$\end{footnotesize}}}$}}}

\newcommand{\stg}
{\mathrel{\raise2pt\hbox{${\mathop>\limits_{\raise1pt\hbox{\mbox{$\sim$}}}}$}}}

\newcommand{\ste}
{\mathrel{\raise2pt\hbox{${\mathop=\limits_{\raise1pt\hbox{\mbox{$\sim$}}}}$}}}

\def \ee{\begin{equation}}
\def \e{\end{equation}}
\def\beq{\begin{eqnarray}}
\def\eq{\end{eqnarray}}
\def\beqx{\begin{eqnarray*}}
\def\eqx{\end{eqnarray*}}
\def\l{\label}  
\renewcommand{\theequation}{\arabic{section}.\arabic{equation}}

\def\omega{\alpha}

\def\0{{\bf 0}}

\def\12{{1\over 2}}
\def\sl{\sup\limits}

\def\ga{\gamma}

\def\chi{{\cal X}}

\begin{document}

\title{New error estimates of the weighted $L^2$ projections}

\author{Qiya Hu and Yuhan Luo}

 \thanks{1. ICMSEC and State Key Laboratory of Mathematical Sciences, Academy of Mathematics and Systems Science, Chinese Academy of Sciences, Beijing
 100190, China; 2. School of Mathematical Sciences, University of Chinese Academy of Sciences, Beijing 100049,
 China (hqy@lsec.cc.ac.cn, luoyuhan@lsec.cc.ac.cn). The work of the authors was supported by the National Natural Science Foundation
of China grant G12571447.}


\maketitle
\begin{abstract}
It is known that the weighted $L^2$ projection operator exhibits approximation properties different from those of the classical $L^2$ projection, in the sense that
the $L^2$ error of the weighted $L^2$ projection of an $H^1$ function generally cannot be bounded by the $H^1$ semi-norm of the function. In this paper, we establish sharper $L^2$
error estimates for the weighted $L^2$ projection of an $H^1$ function under general weight distributions. These new estimates show that the $L^2$ errors of the weighted $L^2$
projection can be controlled by the $H^1$ semi-norm of the function, except when the weight distribution is highly irregular, such as those resembling a ``checkerboard" pattern.
These results can be applied to more refined analyses of domain decomposition methods and multigrid methods for certain partial differential equations with large jump coefficients.
\end{abstract}



\pagestyle{myheadings}
\thispagestyle{plain}
\markboth{}{}


\section{Introduction}\l{sec:introduction}
\setcounter{equation}{0}
Domain decomposition methods and multigrid methods are efficient numerical techniques for solving large-scale systems arising from the discretization of partial differential equations (PDEs). These methods have been widely studied for various PDE models, as documented in the literature (see, e.g., \cite{hipt1998}, \cite{HZ2}, \cite{toselli2004domain}, \cite{Xu-zhu2008}, and references therein).

A fundamental tool in the convergence analysis of these methods is the $L^2$ projection operator from a fine finite element space (or the $H^1$ space) into a coarser finite element space (see \cite{xu1992}). When the coefficients of the underlying PDE exhibit large jumps, it becomes necessary to employ a weighted $L^2$ projection operator, where the weights are defined by these coefficients. This approach is essential for determining whether the convergence rate depends on the magnitude of the coefficient jumps, as illustrated in works such as \cite{Hu2010} and \cite{Xu-zhu2008}.

Consider a bounded Lipschitz domain in three dimensions and a given positive piecewise constant function. The domain is decomposed into subdomains such that the function is constant on each subdomain; these constants are referred to as weights. For applications in domain decomposition methods, let $d$ denote the mesh size that defines the image finite element space of both the classical $L^2$ projection operator and the weighted $L^2$ projection operator, where the latter is defined with respect to weighted inner products.
It is well known that the $L^2$ error of the classical $L^2$ projection of an $H^1$ function can be bounded by the $H^1$ semi-norm of the function, multiplied by a convergence factor $d$. However, this result does not generally hold for the weighted $L^2$ projection. The approximation properties of the weighted $L^2$ projection have been studied in \cite{bramble1991some} and \cite{xu1991counterexamples}. Their results show that the $L^2$ error of the weighted $L^2$ projection of an $H^1$ function is typically controlled by the weighted $H^1$ full-norm rather than the weighted $H^1$ semi-norm of the function, with the convergence factor $d$ replaced by $d |\log d|^{1/2}$.
If the weighted $L^2$ projection operator acts on a finite element space with a finer mesh size $h$, the $L^2$ error can be bounded by the weighted $H^1$ semi-norm, but the convergence factor becomes $d(d/h)^{1/2}$.
When these error estimates for the weighted $L^2$ projection are applied to the convergence analysis of domain decomposition methods with simple coarse spaces and multigrid methods for Laplace-type equations with large jump coefficients, the resulting convergence rates characterized by the condition number of the preconditioned system—are often severely degraded by the coefficient jumps. This outcome contradicts the numerically observed convergence rates in many practical situations. This discrepancy can be explained by the reduced condition number demonstrated in \cite{Hu2010} and \cite{Xu-zhu2008}. We conjecture that the existing error estimates for the weighted $L^2$ projection may not be sharp.

In this paper, motivated by the first author's recent work \cite{Hu2023}, we undertake a more detailed investigation of the weighted $L^2$ projection. The relationship among the weights, which is crucial for convergence, is referred to as the distribution of the weights. We find that the $L^2$ error estimate of the weighted $L^2$ projection depends heavily on this distribution. For a general distribution of weights, we establish new $L^2$ error estimates for the weighted $L^2$ projection. These estimates show that the $L^2$ error can be bounded by a weighted combination of local $H^1$ full-norms or local $H^1$ semi-norms of the given $H^1$ function (with a convergence factor of $d|\log d|^{1/2}$). Notably, the local $H^1$ full-norms are required only for “bad” subdomains corresponding to a thorny weight distribution.

In particular, if the weight distribution satisfies the quasi-monotonicity condition introduced in \cite{Petzoldt2002}, then the $L^2$ error of the weighted $L^2$ projection can be controlled by a weighted $H^1$ semi-norm of the function (with the same convergence factor $d|\log d|^{1/2}$). We also consider the special case where the space on which the weighted $L^2$ projection acts is a finite element space with a refining mesh size $h$; for this setting, we obtain sharper results. Moreover, unlike in previous studies, we do not require the subdomains to be standard polyhedra--they may be general Lipschitz polyhedra associated with the mesh size $h$. In a subsequent paper, we will apply these new estimates to design an economical coarse solver and to establish a sharp convergence analysis of domain decomposition methods for H(\text{curl})-elliptic problems with large jump coefficients.

The paper is organized as follows. Section 2 introduces the domain decomposition based on the coefficient distribution, defines the corresponding weighted $L^2$ projection, and provides the definitions of a Lipschitz polyhedron,
a Lipschitz edge, and a Lipschitz face. In Section 3, we prove an edge lemma and a face lemma for a Lipschitz polyhedron. Section 4 introduces the new concept of thorny edges and presents four new error estimates for the
weighted $L^2$ projection. The proofs of the main results are given in Section 5.

\section{Preliminaries}
\label{sec:Preliminaries}

\subsection{The weighted \( L^2 \) projection}
Let \( \Omega \subset \mathbb{R}^3 \) be a bounded Lipschitz domain.
For a given bounded and positive function \( \alpha(x) \), which corresponds to the coefficient of a partial differential equation,
let \( \Omega \) be partitioned into \( N_0 \) disjoint Lipschitz subdomains:
\begin{equation}
\label{eq:domain decomposition}
\overline{\Omega} = \bigcup_{i=1}^{N_0} \overline{\Omega}_i
\end{equation}
such that \( \alpha(x) \) has small variation on each $\Omega_i$. For convenience, we assume that
\begin{equation}
\alpha(x) = \alpha_i, \quad \forall x \in \Omega_i,
\end{equation}
where \( \alpha_i \) is a positive constant. This decomposition is fully determined by the jump discontinuities of \( \alpha(x) \); hence \( N_0 \) often is a fixed constant and every \( \Omega_i \) has a diameter $O(H)$ ($H$ is a constant).

Let ${\mathcal T}_d$ be a quasi-uniform triangulation on $\Omega$ with the mesh size $d$. We also consider another quasi-uniform triangulation ${\mathcal T}_h$, which is a refinement of ${\mathcal T}_d$. Assume that each $\Omega_i$ is a union of some elements of ${\mathcal T}_d$ (and  ${\mathcal T}_h$).
We use $S_d(\Omega)$ and $S_h(\Omega)$, which are subspaces of $H^1_0(\Omega)$, to denote linear finite element spaces corresponding the triangulation ${\mathcal T}_d$ and ${\mathcal T}_h$, respectively. Then we have $S_d(\Omega)\subset S_h(\Omega)$. Let $Q_d^{\alpha}: L^2(\Omega)\rightarrow S_d(\Omega)$ be the weighted $L^2$ projection operator defined as
$$ (Q_d^{\alpha}u, v)_{L^2_\alpha(\Omega)}=(u, v)_{L^2_\alpha(\Omega)},\quad u\in L^2(\Omega),~\forall v\in S_d(\Omega). $$
where $(\cdot,\cdot)_{L^2_\alpha(\Omega)}$ denotes the weighted $L^2$ inner product:
$$ (u, v)_{L^2_\alpha(\Omega)}=\sum\limits_{i=1}^{N_0}\alpha_i\int_{\Omega_i}u\cdot v dx,\quad u, v\in L^2(\Omega). $$

Let \( \mathcal{N}_h \) denote the set of all mesh nodes in \( \mathcal{T}^h \), and \( \mathcal{E}_h \) and \( \mathcal{F}_h \) denote the sets of all edges and faces of elements in \( \mathcal{T}^h \).
Similarly, we can define \( \mathcal{N}_d \), \( \mathcal{E}_d \) and \( \mathcal{F}_d \).
\subsection{Irregular subdomains}
\label{sec:Irregular Subdomain Decompositions}

In applications we cannot require that each subdomain $\Omega_i$ is just a usual polyhedron (with a fixed number of faces).

$\bullet$ Lipschitz polyhedron. Assume that $\Omega_i$ is a union of elements associated with a triangulation ${\mathcal T}_h$ and is a Lipschitz domain with the Lipschitz constant independent of $h$.
In other words, $\Omega_i$ can be regarded as a perturbation of a usual polyhedron (see Fig. 2.1).

 Since each \( \Omega_i \) is only assumed to be Lipschitz, its faces and edges may have complex, non-planar geometries. We now give precise definitions of such a ``face" and an ``edge".

$\bullet$ Lipschitz edge. Let $E$ be a connected union set of element edges that lie on the boundary of some Lipschitz polyhedron $\Omega_i$. Assume that $E$ has the length $O(H)$. Let $l$ be a straight line determined by two endpoints of $E$.
We use $\delta_E$ to denote the maximal distance from all nodes on $E$ to $l$. If $\delta_E$ is smaller than $c_0h$ with a constant $c_0$, then $E$ is called a Lipschitz edge of $\Omega_i$. In other words, a Lipschitz edge can be viewed as a perturbation
of some edge of a usual polyhedral subdomain.

$\bullet$ Lipschitz face. Let $F$ be a connected union set of element faces that lie on the boundary of some Lipschitz polyhedron $\Omega_i$. Assume that $\partial F$ is a union of several Lipschitz edge of $\Omega_i$.
Let $\pi$ denote the plane determined by three mesh nodes on $F$, and let $\delta_F$ denote the maximal distance from all nodes on $F$ to $\pi$. If $\delta_F$ is smaller than $c_0h$ with a constant $c_0$, then $F$ is called a Lipschitz face of $\Omega_i$. In other words, a Lipschitz face can be viewed as a perturbation
of some face of a usual polyhedral subdomain.


As illustrated in Figure~\ref{fig:edge and face}, the defined Lipschitz face $F$ and Lipschitz edge $E$ can be regarded as  geometric perturbations of a plane polygon and a straight segment with size $O(H)$, respectively, and the perturbation magnitudes are independent of mesh refinement. That is, there exists a constant $C_0 > 0$, independent of the mesh size $h$, such that
\[
\delta_E, \delta_F \leq C_0
\]
for all ``face" $F$ and ``edge" $E$. Throughout this paper, $F$, $E$, and $V$ denote the faces, edges, and vertices of subdomains. Unless otherwise stated, all edges and faces considered below adhere to the above definitions.

\begin{figure}[!htbp]
    \centering
    \includegraphics[width=0.3\textwidth]{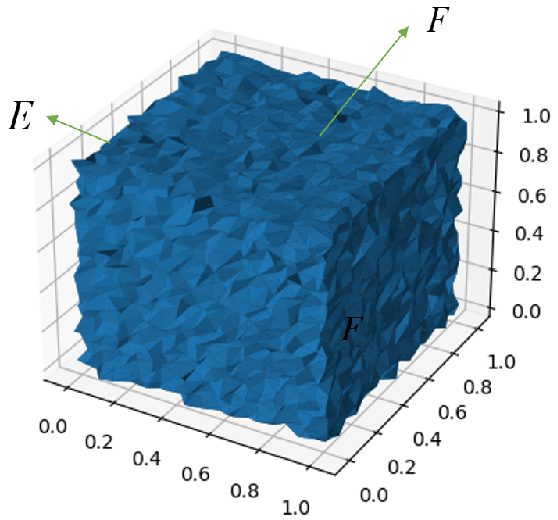}
    \caption{Illustration of faces and edges. Each face \(F\) and edge \(E\) is a connected union of element faces or edges, respectively, and can be viewed as a geometric perturbation of a plane polygon or a straight segment.}
    \label{fig:edge and face}
\end{figure}

\section{Edge Lemma and face Lemma on a Lipschitz polyhedron}
\label{sec:Edge and Face Lemmas on Lipschitz Domains}
The error estimates of the weighted \( L^2 \) projection $Q_d^{\alpha}$ rely on edge and face lemmas for Lipschitz edges and faces. However, existing results cannot be
directly applied to the “edges” and “faces” defined in the previous section. This section, inspired by \cite{bringmans2020discrete}, establishes and proves the edge and face lemmas on a Lipschitz polyhedron accordingly.
The implicit constants in the edge lemma and face lemma depend on the perturbation magnitude: the larger the deviation $C_0$ from a plane or line is, the larger the constants are.

In this section we use $G\subset \Omega$ to denote a general Lipschitz polyhedron $\Omega_i$ and consider a general mesh size $h$ (it can be also chosen as $d$ in subsection 2.1). Let $\tilde{S}_h(G)\subset H^1(G)$ be
the linear finite element space associated with a triangulation on $G$ with the mesh size $h$.
For any subset $K$ of $G$, we use \( \mathcal{N}_h(K) \) to denote the set of mesh nodes on $K$, and $\tilde{S}_h(K)\subset\tilde{S}_h(G)$ to denote the linear finite element space on $K$.
For a node $x_k$ on $G$, let $\varphi_k\in \tilde{S}_h(G)$ denote the nodal basis function corresponding to $x_k$. For a subset $K$ of $G$, define the restriction operator \( I_K^0: \tilde{S}_h(G) \to \tilde{S}_h(K) \) by
\[
I_K^0 u^h := \sum_{x_k \in \mathcal{N}_h(K)} u^h(\mathbf{x}_k)\, \varphi_k.
\]

\begin{lemma}[Edge Lemma]
\label{lem:edge lemma}
Let \( G \subset \mathbb{R}^3 \) be a Lipschitz polyhedron, equipped with a family of quasi-uniform tetrahedral meshes \( \{\mathcal{T}^h\} \), and assume that \( \operatorname{diam}(G) = H \). Let \( E \) be a Lipschitz edge
(as defined in Section~\ref{sec:Irregular Subdomain Decompositions}) on \( \partial G \). Then, for any \( u^h \in \tilde{S}_h(\partial G) \), the following estimate holds:
\[
\left\| I_E^0 u^h \right\|_{H^{1/2}(\partial G)} \leq C_1 \left\| u^h \right\|_{L^2(E)} \leq C_2 \Big(\log \frac{H}{h}\Big)^{1/2} \left\| u^h \right\|_{H^{1/2}(\partial G)},
\]
where \( C_1, C_2 > 0 \) are constants independent of \( h \) and \(H\), but depend on the mesh quasi-uniformity parameter.
\end{lemma}

\begin{lemma}[Face Lemma]
\label{lem:face lemma}
Let \( G \subset \mathbb{R}^d \) be a Lipschitz polyhedron, equipped with a family of quasi-uniform tetrahedral meshes \( \{ \mathcal{T}^h \} \), and assume that \( \operatorname{diam}(G) = H \). Let \( F \) be a Lipschitz face on \( \partial G \) when \( d = 3 \), or an edge when \( d = 2 \), both as defined in Section~\ref{sec:Irregular Subdomain Decompositions}. Then, for any \( u^h \in \tilde{S}_h(\partial G) \), we have
\[
\| I_F^0 u^h \|_{H^{1/2}(\partial G)} \leq C \Big(\log \frac{H}{h}\Big) \| u^h \|_{H^{1/2}(\partial G)},
\]
where \( C > 0 \) is a constant independent of \( h \) and \(H\), but may depend on the perturbation constant \( C_0 \) and mesh quasi-uniformity parameter.
\end{lemma}

\subsection{Proofs of the edge and face Lemmas}
This section adopts the centroid-slicing approach proposed in \cite{bringmans2020discrete} to prove the edge and face lemmas on Lipschitz domains. The main change is to replace the original arguments that rely on planar geometry by a framework based on discrete \( L^2 \) norms, thereby eliminating the dependence on polyhedral subdomain structures.

We begin by introducing the geometric definition of centroid slices and the corresponding inverse inequalities. Let \( F \subset \partial G \) be an arbitrary face, \( E \) an edge of \( F \), and let \( v^h \in \tilde{S}_h(F) \). The basic idea is to classify the elements touching the face $F$ by the positions of their barycenters. Without loss of generality, we use a rectangle to replace a Lipschitz face $F$ in the diagram Figure~\ref{fig:edge and face lemma}.  Following Lemmas 4.24 and 4.26 in \cite{bringmans2020discrete}, the set
\[
\mathcal{T}^h(F) := \{ K \in \mathcal{T}^h \mid K \text{ has at least one face on } F \}
\]
is partitioned into subsets \( \{\Lambda_i\}_{i=1}^N \) defined by
\[
\Lambda_i := \left\{ K \in \mathcal{T}^h(F) \mid (i-1)h < \Pi(c_K) \leq ih \right\},
\]
where \( c_K \) denotes the centroid of element \( K \), and \( \Pi \) is the orthogonal projection onto the \( y \)-axis.
The integer \( N \) is the minimal number such that all \( K \in \mathcal{T}^h(F) \) are classified.

\begin{figure}[!htbp]
    \centering
    \includegraphics[width=0.9\textwidth]{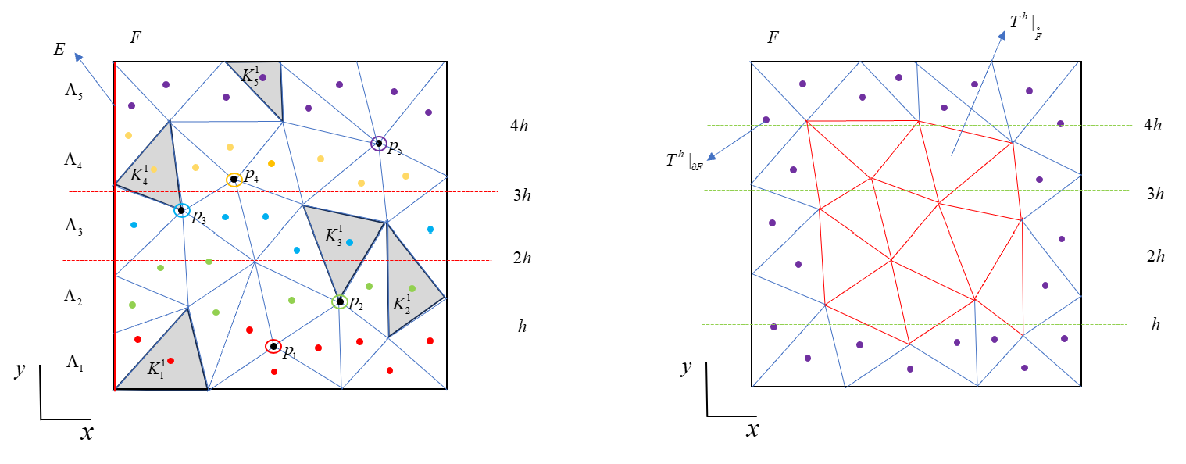}
    \caption{Illustration of centroid slices. Left: schematic of the selected nodes \( \{\textbf{p}_i\}_{i=1}^N \) and elements \( \mathcal{T}_{\partial_1 v}^h(F) = \{K^1_i\}_{i=1}^N \); Right: mesh subsets \( \mathcal{T}^h|_{\mathring{F}} \) and \( \mathcal{T}^h|_{\partial F} \).}
    \label{fig:edge and face lemma}
\end{figure}

Let \( \mathcal{N}_h(\Lambda_i) \) and \( \mathcal{N}_h(F) \) denote the sets of nodes in \( \Lambda_i \) and on \( F \), respectively. For each \( i = 1, \dots, N \), choose a node \( \mathbf{p}_i \in F \) such that
\begin{equation}
\label{eq:p_i}
v^h(\mathbf{p}_i)^2 = \max_{\mathbf{p} \in \mathcal{N}_h(\Lambda_i) \cap \mathcal{N}_h(F)} v^h(\mathbf{p})^2.
\end{equation}
Since \( v^h \in \tilde{S}_h(F) \) is piecewise linear, the maximum of \( |v^h| \) on any triangular face of \( F \) is attained at a node. By construction, for any node \( \mathbf{p} \in E \cap \mathcal{N}_h(\Lambda_i) \), we have \( v^h(\mathbf{p})^2 \leq v^h(\mathbf{p}_i)^2 \). Therefore, the discrete \( L^2 \) norm on edge \( E \) can be controlled by the nodal values at \( \{\mathbf{p}_i\} \), as will be shown in the proof of the edge lemma below.

Let \( \hat{v}^h \in \tilde{S}_h(G) \) be the generalized discrete harmonic extension of \( v^h \). For each \( i = 1, \dots, N \) and each direction \( j = 1,2,3 \), select an element \( K_i^j \in \Lambda_i \subset \mathcal{T}^h(F) \) such that
\begin{equation}
\label{eq:T_partial_i v^h(F)}
|\partial_j \hat{v}^h(c_{K_i^j})|^2 = \max_{K \in \Lambda_i} |\partial_j \hat{v}^h(c_K)|^2,
\end{equation}
and define \( \mathcal{T}_{\partial_j v}^h(F) := \{K_i^j\}_{i=1}^N \) for \( j = 1,2,3 \). Since \( \hat{v}^h \in \tilde{S}_h(G) \) is piecewise linear, each partial derivative \( \partial_j \hat{v}^h \) is constant on every element \( K \in \mathcal{T}^h(F) \). Hence, for each \( \Lambda_i \), there exists an element where \( |\partial_j \hat{v}^h| \) attains its maximum over \( \Lambda_i \), allowing us to define the sets \( \mathcal{T}_{\partial_j v}^h(F) \) for \( j = 1,2,3 \).

As illustrated in the left panel of Figure~\ref{fig:edge and face lemma}, the set \( \mathcal{T}^h(F) \) is partitioned into five subsets. In each \( \Lambda_i \), a node \( \mathbf{p}_i \in \mathcal{N}_h(F) \) and an element \( K_i^1 \in \Lambda_i \) are selected to satisfy \eqref{eq:p_i} and \eqref{eq:T_partial_i v^h(F)}, respectively. The selections of \( \{\mathbf{p}_i\}_{i=1}^N \) and \( \mathcal{T}_{\partial_j v}^h(F) \) are not unique and depend on the function \( v \), but this does not affect the subsequent analysis.

We next present inverse inequalities associated with the sets \( \{\mathbf{p}_i\}_{i=1}^N \) and \( \mathcal{T}_{\partial_i v}^h(F) \) for \( i = 1,2,3 \).

\begin{lemma}
\label{lem:3.2}
Let \( G \subset \mathbb{R}^3 \) be a Lipschitz domain with \( \mathrm{diam}(G) = H \), equipped with a family of quasi-uniform tetrahedral meshes \( \{\mathcal{T}^h\} \). Let \( F \subset \partial G \) be an arbitrary face, and let \( \{\mathbf{p}_i\}_{i=1}^N \) be the set of nodes defined in \eqref{eq:p_i}. Then, for any \( v^h \in \tilde{S}_h(\partial G) \), the following estimate holds:
\[
    h\sum_{i=1}^N v^h(\mathbf{p}_i)^2\leq C\big( \log \frac{H}{h} \big) \left\| v^h \right\|_{H^{1/2}(\partial G)}^2,
\]
where the constant \( C > 0 \) is independent of \( h \), \( H \), and \( v \).

\end{lemma}

\begin{proof} The details can be found in the proof of Lemma 5.8 in  \cite{bringmans2020discrete}. The proof is based solely on element-wise estimates and summations, without relying on geometric assumptions specific to polyhedral subdomains, and thus extends naturally to Lipschitz domains. The conclusion follows from the properties of the generalized discrete harmonic extension, which carry over to Lipschitz domains via the Scott--Zhang interpolation operator (see Section~4 in \cite{scott1990finite}).
\end{proof}

\begin{lemma}
\label{lem:face lemma3}
Let \( G \subset \mathbb{R}^3 \) be a Lipschitz domain with \( \mathrm{diam}(G) = H \), equipped with a family of quasi-uniform tetrahedral meshes \( \{\mathcal{T}^h\} \). Let \( F \subset \partial G \) be an arbitrary face, and let the sets \( \mathcal{T}_{\partial_i v}^h(F) \) be defined as in \eqref{eq:T_partial_i v^h(F)}. For any \( v^h \in \tilde{S}_h(\partial G) \), let \( \hat{v}^h \in \tilde{S}_h(G) \) denote the generalized discrete harmonic extension of \( v^h \). Then, for \( i = 1,2,3 \), the following estimate holds:
\[
h^3\sum_{K\in\mathcal{T}_{\partial_i v}^h(F)}|\partial_i \hat{v}^h(c_K)|^2 \leq C \big(\log \frac{H}{h}\big) \|v^h\|^2_{H^{1/2}(\partial G)},
\]
where the constant \( C > 0 \) is independent of \( h \), \(H\), and \( v \).
\end{lemma}

\begin{proof}
The details can be found in the proof of Corollary 4.32 in \cite{bringmans2020discrete}. The argument follows the same strategy as in Lemma~\ref{lem:3.2}, with the difference that it relies on Lemma 4.30 in \cite{bringmans2020discrete}, which is established at the element level and does not depend on any geometric assumptions specific to polyhedral subdomains.
\end{proof}

Based on the above lemma, we now present a proof of Lemma~\ref{lem:edge lemma}.
\renewcommand{\proofname}{Proof of the edge lemma}
\begin{proof}
For the first inequality, we refer to the proof of Lemma 4.9 in \cite{xu1998some}. For the second inequality, using the discrete \( L^2 \) norm, we have
\begin{equation}
\label{eq:edge lemma 001}
  \left\|u^h\right\|_{L^2(E)}^2 \stackrel{\sim}{=} h \sum_{\mathbf{p} \in \mathcal{N}_h(E)} u^h(\mathbf{p})^2 \leq h \sum_{i=1}^N \sum_{\mathbf{p} \in \mathcal{N}_h(E) \cap \mathcal{N}_h(\Lambda_i)} u^h(\mathbf{p}_i)^2,
\end{equation}
where \( \mathcal{N}_h(E) \) denotes the set of nodes on the edge \( E \), and \( \{\mathbf{p}_i\}_{i=1}^N \) is defined as in \eqref{eq:p_i}. By the quasi-uniformity of the mesh, we obtain
\[
\sum_{\mathbf{p} \in \mathcal{N}_h(E) \cap \mathcal{N}_h(\Lambda_i)} 1 \lesssim \frac{h}{\min\limits_{K \in \mathcal{T}^h} \rho_K} \lesssim c_0 c_1.
\]
Hence, by Lemma~\ref{lem:3.2}, it follows that
\[
\left\|u^h\right\|_{L^2(E)}^2 \lesssim h \sum_{i=1}^N u^h(\mathbf{p}_i)^2 \lesssim \big( \log \frac{H}{h} \big) \left\| u^h \right\|_{H^{1/2}(\partial G)}^2,
\]
where the hidden constant depends only on the quasi-uniformity of the mesh.
\end{proof}

We proceed to prove Lemma~\ref{lem:face lemma}, building upon Lemma~\ref{lem:edge lemma}.
\renewcommand{\proofname}{Proof of the face lemma}

\begin{proof}
We first apply the norm equivalence in the Lions--Magenes space to expand the left-hand side of the inequality:
\begin{equation}
\label{eq:face lemma1}
\|I^0_F u^h\|_{H^{1/2}(\partial G)}^2 \stackrel{\sim}{=} \|I^0_F u^h\|_{H^{1/2}_{00}(F)}^2 \stackrel{\sim}{=} |I^0_F u^h|_{H^{1/2}(F)}^2 + \int_F \frac{(I^0_F u^h)^2(\mathbf{x})}{d(\mathbf{x}, \partial F)} \, ds(\mathbf{x}).
\end{equation}
The first term on the right-hand side can be estimated as in Section 4.3 of \cite{xu1998some}. It suffices to estimate the second term. 
Let \( \mathcal{T}^h|_{F} \) denote the mesh on \( F \) induced by \( \mathcal{T}^h \), and $\hat{F}$ denote the union of all the elements in the interior of $F$. Define
\[
\mathcal{T}^h|_{\partial F} := \{T \in \mathcal{T}^h|_{F} \mid T \cap \partial F \neq \emptyset\}, \quad
\mathcal{T}^h|_{\mathring{F}} := \{T \in \mathcal{T}^h|_{F} \mid T \cap \partial F = \emptyset\}.
\]
As shown in Figure~\ref{fig:edge and face lemma}, we have \( \mathcal{T}^h|_{F} = \mathcal{T}^h|_{\mathring{F}} \cup \mathcal{T}^h|_{\partial F} \).

For an element $T$ on $F$, we use \( D_T \) to denote the projection of \( T \) onto the \( xoy \)-plane (projections onto the \( yoz \) or \( xoz \)-plane may also be used). According to the definition of a Lipschitz  face in Section~\ref{sec:Irregular Subdomain Decompositions}, we obtain
\begin{equation}
\label{eq:3.21}
\begin{aligned}
&\int_F \frac{(I^0_F u^h)^2(\mathbf{x})}{d(\mathbf{x}, \partial F)} \, ds(\mathbf{x})
\leq C(\delta_F) \sum_{T \in \mathcal{T}^h|_{F}} \int_{D_T} \frac{(I^0_F u^h)^2(\mathbf{x})}{d(\mathbf{x}, \partial F)} \, dx \, dy \\
&\lesssim \sum_{T \in \mathcal{T}^h|_{\partial F}} \int_{D_T} \frac{(I^0_F u^h)^2(\mathbf{x})}{d(\mathbf{x}, \partial F)} \, dx \, dy + \sum_{T \in \mathcal{T}^h|_{\mathring{F}}} \int_{D_T} \frac{(I^0_F u^h)^2(\mathbf{x})}{d(\mathbf{x}, \partial F)} \, dx \, dy =: I_1 + I_2.
\end{aligned}
\end{equation}
Here \( C(\delta_F) \) is a constant depending on the flatness parameter \( \delta_F \), but independent of \( h \) and \(H\). When \( \delta_F \) increases, \( C(\delta_F) \) also increase. Since \( \delta_F \) is bounded,
\( C(\delta_F) \) is a constant.

For simplicity, let \( w(x,y,z) := I^0_F u^h(x,y,z) \). We begin with estimating \( I_1 \). Fix \( T \in \mathcal{T}^h|_{\partial F} \), and let \( K_T \in \mathcal{T}^h \) be the tetrahedral element having \( T \) as one of its faces. For a fixed \( \mathbf{x} \in T \), there exists a point \( \mathbf{x}_0 = (x_0, y_0, z_0) \in \partial F \) such that \( d(\mathbf{x}, \partial F) = d(\mathbf{x}, \mathbf{x}_0) \). Since \( w(\mathbf{x}_0) = 0 \) (by the definition of \( I^0_F \)), we have
\begin{equation}
\label{eq:3.23}
\begin{aligned}
\frac{w(\mathbf{x})^2}{d(\mathbf{x}, \partial F)} &= \frac{\big( \int_{x_0}^x \partial_1 w(t, y) \, dt + \int_{y_0}^y \partial_2 w(x_0, t) \, dt \big)^2}{\sqrt{(x - x_0)^2 + (y - y_0)^2 + (z - z_0)^2}} \\
&\lesssim h \max_{t \in [x_0, x]} |\partial_1 w(t, y)|^2 + h \max_{t \in [y_0, y]} |\partial_2 w(x_0, t)|^2.
\end{aligned}
\end{equation}
Substituting this into \( I_1 \) and applying Lemma~\ref{lem:face lemma3}, we obtain
\begin{equation*}
\begin{aligned}
 I_1 = \sum_{i=1}^N \sum_{K_T \in \Lambda_i} \int_{D_T} \frac{w(\mathbf{x})^2}{d(\mathbf{x}, \partial F)} \, dx \, dy
&\lesssim h^3 \sum_{i=1}^N \sum_{j=1}^3 \big( \partial_j \hat{u}^h(c_{K^j_i}) \big)^2
\\&\lesssim \big( \log \frac{H}{h} \big) \|u^h\|_{H^{1/2}(\partial G)}^2,
\end{aligned}
\end{equation*}
where the sets \( \{K^j_i\}_{i=1}^N \) for \( j = 1,2,3 \) are defined in \eqref{eq:T_partial_i v^h(F)}.

We next estimate \( I_2 \). For each \( \Lambda_i \), classify the elements \( T \in \mathcal{T}^h|_{\mathring{F}} \) with \( K_T \in \Lambda_i \) into layers as follows:
\begin{gather*}
\Lambda_{i,N_i} := \left\{ T \in \mathcal{T}^h|_{\mathring{F}}\;\middle|\; d_x(\mathbf{x}, \partial F) > \rho N_i h,\ \forall \mathbf{x} \in T \right\}, \\
\Lambda_{i,j} := \left\{ T \in \mathcal{T}^h|_{\mathring{F}}\;\middle|\; d_x(\mathbf{x}, \partial F) > \rho j h,\ \forall \mathbf{x} \in T \right\} \setminus \Lambda_{i,j+1}, \, j = N_i - 1, \dots, 1,
\end{gather*}
where \( \rho = 1/(c_0 c_1) \), and \( d_x(\mathbf{x}, \partial F) \) denotes the projection distance from \( \mathbf{x} \) to \( \partial F \) along the \( x \)-axis. The integer \( N_i \) is the largest \( j \) such that \( \rho j h < H/2 \), and hence \( N_i \lesssim H/h \). Then,
\begin{equation}
\label{eq:3.26}
\begin{aligned}
I_2 &\lesssim \sum_{i=1}^N \sum_{j=1}^{N_i} \sum_{T \in \Lambda_{i,j}} \int_{D_T} \frac{u^h(\mathbf{p}_i)^2}{d(\mathbf{x}, \partial F)} \, dx \, dy \\
&\lesssim \sum_{i=1}^N u^h(\mathbf{p}_i)^2 \sum_{j=1}^{N_i} \sum_{T \in \Lambda_{i,j}} \int_{D_T} \frac{1}{d(\mathbf{x}, \partial F)} \, dx \, dy,
\end{aligned}
\end{equation}
where \( \{\mathbf{p}_i\}_{i=1}^N \) are defined in \eqref{eq:p_i}. Using the standard bound for harmonic series, we have
\[
\sum_{j=1}^{N_i} \sum_{T \in \Lambda_{i,j}} \int_{D_T} \frac{1}{d(\mathbf{x}, \partial F)} \, dx \, dy
\lesssim \sum_{j=1}^{N_i} h^2 \cdot \frac{1}{\rho j h} \lesssim h \log N_i \lesssim h \log \big( \frac{H}{h} \big).
\]
Substituting into \eqref{eq:3.26} and applying Lemma~\ref{lem:3.2}, we obtain the estimate for \( I_2 \):
\begin{equation}
\label{eq:3.27}
I_2 \lesssim h \log \big( \frac{H}{h} \big) \sum_{i=1}^N u^h(\mathbf{p}_i)^2
\lesssim \big( \log \frac{H}{h} \big)^2 \|u^h\|_{H^{1/2}(\partial G)}^2,
\end{equation}
where the hidden constant depends only on \( \delta_F \) and the quasi-uniformity of the mesh. For the two dimensional case, with \( F \) denoting a Lipschitz edge, we can use the same argument.
\end{proof}

\section{Error estimates of the weighted \( L^2 \) projection associated with an irregular subdomain decomposition}
\label{sec:Error Estimates for Weighted L2 Projections}
In this section we present new error estimates of the weighted \( L^2 \) projection $Q_d^{\alpha}$ associated with an irregular subdomain decomposition $\overline{\Omega} = \cup_{i=1}^{N_0} \overline{\Omega}_i$ (see Section 2).
In the rest of this paper, two triangulations ${\mathcal T}_d$ and ${\mathcal T}_h$ will be involved, where ${\mathcal T}_h$ is a refinement of ${\mathcal T}_d$. Notice that each subdomain $\Omega_i$ is a union of some elements ${\mathcal T}_d$, then $\Omega_i$ is also a union of some elements ${\mathcal T}_h$.

As we will see, under the quasi-monotonicity condition introduced in \cite{Petzoldt2002}, the $L^2$ errors of such projections
can be controlled by weighted \( H^1 \) semi-norm of the considered function. However, if {\it thorny} vertices (see \cite{Hu2023}) or {\it thorny} edges appear, the errors should be controlled by weighted \( H^1 \) full-norm instead of semi-norm. In particular, new estimates are also established for functions in the refined finite element space \( S_{h}(\Omega) \).

\subsection{Some notions}

Let \( V \) and \( E \) denote a vertex and an edge of some subdomain $\Omega_i$, respectively. Let \( \mathscr{S}_V \) be the set of subdomains having \( V \) as a vertex, and \( \mathscr{S}_E \) be the set of subdomains having \( E \) as an edge. We first introduce the notions of thorny vertex, thorny edge and the quasi-monotonicity condition. Following Section 3.1 of \cite{Hu2023}, define the upper intersection of a subdomain \( \Omega_k \) as
\[
\mathcal{S}_k := \bigcup_{\alpha_l \geq \alpha_k,\, l \neq k} \big( \partial \Omega_k \cap \partial \Omega_l \big) \cup \big( \partial \Omega_k \cap \partial \Omega \big).
\]

The definition of a {\it thorny vertex} was given by Definition 3.2 in \cite{Hu2023}. Here we define analogously a {\it thorny edge}. If there exists \( \Omega_k \in \mathscr{S}_E \) such that \( E \) is an isolated edge in \( \mathcal{S}_k \), i.e., \( E \) is not contained in any face of \( \mathcal{S}_k \), then \( E \) is called a {\it thorny edge}. Then we easily know the definition of a {\it thorny vertex}.

\begin{remark}
The number and location of thorny edges and thorny vertices are determined by the distribution of the discontinuous coefficient \( \alpha(\mathbf{x}) \). We can naturally assume that the number and of thorny edges and thorny vertices is a finite number (independent of $d$ and $h$).
\end{remark}

We  define two subdomain sets associated with a thorny edge \( E \).  If \( \Omega_r \in \mathscr{S}_E^* \), then \( E \) is an isolated edge in \( \mathcal{S}_r \); define \( \mathscr{S}_E^c := \mathscr{S}_E \setminus \mathscr{S}_E^* \). Similarly, we define \( \mathscr{S}_V^* \) and \( \mathscr{S}_V^c \) for a thorny vertex \( V \). The properties of \( \mathscr{S}_E^* \) and \( \mathscr{S}_E^c \) are given below, and analogous results hold for \( \mathscr{S}_V^* \) and \( \mathscr{S}_V^c \) (see Proposition 3.1 in \cite{Hu2023}).

{\bf Proposition 4.1}
Let \( E \) be a thorny edge and \( \Omega_r \in \mathscr{S}_E^* \). If \( \Omega_l \in \mathscr{S}_E \) shares a common face with \( \Omega_r \) that contains \( E \), then \( \Omega_l \in \mathscr{S}_E^c \) and \( \alpha_l < \alpha_r \).

Following Section 4.1 of \cite{Petzoldt2002}, we present the definition of the quasi-monotonicity condition. It requires that the coefficient distribution \( \{\alpha_k\}_{k=1}^{N_0} \) produces no thorny vertex or thorny edge; that is, for each subdomain \( \Omega_k \), the upper intersection \( \mathcal{S}_k \) is either empty or a union of faces of \( \Omega_k \).

Theorem 2.1 in \cite{xu1991counterexamples} provides a counterexample showing that a uniform \( H^1 \) semi-norm error bound for the weighted \( L^2 \) projection may fail in the presence of thorny vertices or edges. This work revisits the problem by establishing an \( H^1 \) semi-norm estimate under the quasi-monotonicity condition and quantifying the effect of thorny structures on the projection error in the general case.

\subsection{The case without thorny vertex}
In this subsection, we establish a new error estimate for the weighted \( L^2 \) projection under the quasi-monotonicity condition. When no thorny vertex appears but thorny edges present, we also provide a new error estimate of the weighted \( L^2 \) projection for functions in a refined finite element space \( S_{h}(\Omega) \).

Let $\|\cdot\|_{L^2_\alpha(\Omega)}$ be the weighed $L^2$ norm induced by $(\cdot,\cdot)_{L^2_\alpha(\Omega)}$. Similarly, let $|\cdot|_{H^1_\alpha(\Omega)}$ denote the weighed $H^1$ semi-norm.
\begin{theorem}
\label{thm: quasi-monotonicity}
Assume that the domain decomposition \eqref{eq:domain decomposition} satisfies the quasi-monotonicity condition.
Then, for any \( u \in H_0^1(\Omega) \), we have the uniform error estimate controlled by $H^1$ semi-norm:
\[
\| (I - Q_d^\alpha) u \|_{L^2_\alpha(\Omega)} \leq C d  | u |_{H^1_\alpha(\Omega)},
\]
where the constant \( C \) depends only on the distribution of the coefficients \( \{\alpha_i\}_{i=1}^{N_0} \), but is independent of the variations of the coefficients and the mesh size \( d \).
\end{theorem}

The above theorem improves upon Theorems 4.3 and 4.7 in \cite{bramble1991some} in the sense that weaker assumptions are required. Unlike Theorem 4.7 in \cite{bramble1991some}, which requires each subdomain to intersect \( \partial \Omega \) with positive $2$-dimensional Lebesgue measure, our result only assumes that the decomposition \eqref{eq:domain decomposition} contains no thorny vertex or thorny edge, making it more applicable to irregular decompositions.

The above result holds for general \( H^1 \) functions. For functions in a refined finite element space \( S_h(\Omega) \), the absence of thorny vertex alone suffices to ensure a quasi-uniform error estimate controlled by $H^1$ semi-norm, even if thorny edges are present, as shown in the next theorem.

\begin{theorem} 
\label{thm:the case satisfying no thorny vertex on S_h(Omega) refine}
Assume that  the  decomposition \eqref{eq:domain decomposition} contains no thorny vertex. Then, for any \( u \in S_{h}(\Omega) \), we have the quasi-uniform error estimate
\begin{equation*}
\| (I - Q_d^\alpha) u \|_{L^2_\alpha(\Omega)} \leq C d \log(H/d)^{1/2}\log(H/h)^{1/2}  | u |_{H^1_\alpha(\Omega)},
\end{equation*}
where the constant \( C \) depends only on the distribution of the coefficients \( \{\alpha_i\}_{i=1}^{N_0} \), but is independent of the variations of the coefficients, the mesh sizes $d$ and \( h \).
\end{theorem}

This result can be viewed as a special case of Theorem~\ref{thm: quasi-monotonicity} for finite element functions, or as an improvement over Theorem~2.2 in \cite{xu1991counterexamples}. The assumption here is weaker, requiring only the absence of thorny vertex, whereas Theorem~2.2 in \cite{xu1991counterexamples} assumes that \( \mathrm{meas}_1(\partial \Omega_i \cap \partial \Omega) > 0 \) holds for all \( i \).

\subsection{General case}\label{sec:main results for error estimates}
In this subsection, we provide error estimates for the weighted \( L^2 \) projection under the general setting where both thorny edges and thorny vertices may appear. The results are given separately for general \( H^1 \) functions and for the special case \( u \in S_{h}(\Omega) \).

We now define two subdomain sets that partition \( \Omega \) into two parts: one requiring special treatment due to the presence of thorny vertices or edges, and the other satisfying the quasi-monotonicity condition. Let \( \mathcal{V}^* \) and \( \mathcal{E}^* \) denote the sets of all thorny vertices and thorny edges, respectively. Define
\begin{equation}
\label{eq:s*}
\mathscr{S}^* = \big( \bigcup_{V\in \mathcal{V}^*} \mathscr{S}_V^* \big) \cup \big( \bigcup_{E\in \mathcal{E}^*} \mathscr{S}_E^* \big), \quad \mathscr{S}_*^c = \left\{ \Omega_{k} \right\}_{k=1}^{N_0} \setminus \mathscr{S}^*.
\end{equation}

\begin{theorem}
\label{thm:thorny vertex and edge}
For any \( u \in H_0^1(\Omega) \), the following estimate holds:
\begin{equation*}
\| (I - Q_d^\alpha) u \|^2_{L^2_\alpha(\Omega)} \leq C d^2\big(\log(H/d)\sum_{\Omega_k\in\mathscr{S}^*} \alpha_k \|u\|^2_{H^1(\Omega_k)} + \sum_{\Omega_k\notin\mathscr{S}^*} \alpha_k |u|^2_{H^1(\Omega_k)}\big),
\end{equation*}
where the constant \( C \) depends only on the distribution of the coefficients \( \{\alpha_i\}_{i=1}^{N_0} \), but is independent of the variations of the coefficients and the mesh size \( d \).
\end{theorem}

The above theorem indicates that, within subdomains containing thorny vertices or thorny edges, the error of the weighted \( L^2 \) projection for \( u \in H_0^1(\Omega) \) increases significantly. The estimate requires to be controlled by the full \( H^1 \)-norm and involves an additional logarithmic factor. In contrast, within subdomains satisfying the quasi-monotonicity condition, the projection error remains controlled by the \( H^1 \) semi-norm. Therefore, the presence of thorny vertices and edges is a critical factor affecting the approximation properties of the weighted \( L^2 \) projection.

We now present the error estimate for the weighted \( L^2 \) projection of functions in \( S_{h}(\Omega) \) in the general case with thorny vertices. Since only thorny vertices affect the estimate, the subdomain sets are redefined to exclude contributions from thorny edges. Analogous to the discussion for general \( H^1 \) functions, we first define two sets:
\begin{equation}
\label{eq:tilde s*}
\mathscr{\tilde{S}}^* = \bigcup_{V\in \mathcal{V}^*} \mathscr{S}_V^* , \quad \mathscr{\tilde{S}}_*^c = \left\{ \Omega_{k} \right\}_{k=1}^{N_0} \setminus \mathscr{\tilde{S}}^*.
\end{equation}

\begin{theorem}
\label{thm:thorny vertex and edge refine} Let \( S_{h}(\Omega) \) be a refined finite element space of $S_h(\Omega)$.
For any \( u \in S_{h}(\Omega) \), we have
\begin{equation*}
\| (I - Q_d^\alpha) u \|^2_{L^2_\alpha(\Omega)} \leq Cd^2\log(H/d)\log(H/h)\big(\sum_{\Omega_k\in\mathscr{\tilde{S}}^*}\alpha_k\|u\|^2_{H^1(\Omega_k)}+\sum_{\Omega_k\notin\mathscr{\tilde{S}}^*}\alpha_k|u|^2_{H^1(\Omega_k)}\big),
\end{equation*}
where the constant \( C \) depends only on the distribution of the coefficients \( \{\alpha_k\}_{k=1}^{N_0} \), but is independent of the variations of the coefficients, the mesh sized $d$ and \( h \).
\end{theorem}

This theorem shows that in subdomains containing thorny vertices, the weighted \( L^2 \) projection error for \( u \in S_{h}(\Omega) \) is significantly increased and must be controlled by the full \( H^1 \) norm. In contrast, if there is no thorny vertex (there may be thorny edges), the projection error remains controlled by the \( H^1 \) semi-norm. Therefore, the presence of thorny vertices is a critical factor affecting the approximation properties of the weighted \( L^2 \) projection.

Theorem~\ref{thm:thorny vertex and edge refine} is a special case of Theorem~\ref{thm:thorny vertex and edge} for finite element functions; both highlight the influence of thorny geometric structures on the weighted \( L^2 \) projection error. Notably, Theorem~\ref{thm:thorny vertex and edge refine} shows that for functions in \( S_{h}(\Omega) \), only thorny vertices significantly affect the projection accuracy, while thorny edges do not.

\section{Proof of the main results}
This section provides the proofs of the main results. We first consider the case without thorny vertex, followed by the general case involving thorny vertices.

\subsection{Proof for the case without thorny vertex}

In this subsection, we prove Theorem~\ref{thm: quasi-monotonicity} and Theorem~\ref{thm:the case satisfying no thorny vertex on S_h(Omega) refine} in sequence. As the proof of the latter closely follows that of the former, with differences arising only in the treatment of thorny edges, we provide details only for the handling of thorny edges in the proof of Theorem~\ref{thm:the case satisfying no thorny vertex on S_h(Omega) refine}, and refer to the proof of Theorem~\ref{thm: quasi-monotonicity} for the remaining arguments.

We begin by introducing the concept of a multilayer model, as in Section~5.1.1 of~\cite{Hu2023}. The subdomains \(\{\Omega_k\}_{k=1}^{N_0}\) are partitioned into \(m\) disjoint nonempty subsets \(\Sigma_1, \Sigma_2, \dots, \Sigma_m\), satisfying: 1) Subdomains in the same subset do not intersect; 2) If \(\Omega_{k_l} \in \Sigma_l\), \(\Omega_{k_j} \in \Sigma_j\) with \(l < j\), and $\partial\Omega_{k_l} \cap \partial\Omega_{k_j} \neq \emptyset$, then \(\alpha_{k_l} \geq \alpha_{k_j}\). The number of layers \(m\) depends only on the distribution of \(\{\alpha_k\}\). For convenience, we assume that the layer number $m$ is a constant independent of $d$ (and $h$), which is satisfied in the most applications.

Let \(\Gamma_{k,l}\) denote the interface between two neighboring subdomains \(\Omega_k\) and \(\Omega_l\), i.e.,
\begin{equation}
\partial\Omega_k\cap\partial\Omega_l=\Gamma_{k,l},
\end{equation}
Define the index set \(\Lambda_k^j\) as the set of all indices \(l\) such that \(\Omega_l\) intersects \(\Omega_k\) and \(\Omega_l \in \Sigma_j\), that is,
\begin{equation}
\label{eq:lambda_k^j}
\Lambda_k^j=\{l:\partial\Omega_l\cap\partial\Omega_k\neq\emptyset,\text{ }\Omega_l\in\Sigma_j\}.
\end{equation}

The following lemma is a direct result of the definitions of the quasi-monotonicity condition and the sets $\{\Sigma_l\}$.
\begin{lemma}
\label{lem:sigma_lsigma_l+1}
Assume that the decomposition \eqref{eq:domain decomposition} satisfies the quasi-monotonicity condition. Let $l \geq 2$, and suppose $\Omega_i \in \Sigma_{l-1}$ and $\Omega_j \in \Sigma_l$ with $\Gamma_{i,j} \neq \emptyset$.
Then $\Gamma_{i,j}$ must be a face.
\end{lemma}

We further have the following result.
\begin{lemma}
\label{lem:sigma_lsigma_p p<l-1}
Assume that the decomposition \eqref{eq:domain decomposition} satisfies the quasi-monotonicity condition. Let $n \geq 3$, and suppose $\Omega_l \in \Sigma_p$ with $p \leq n - 2$ and $\Omega_k \in \Sigma_n$. If $\Gamma_{k,l}$ is an edge $E$ or a vertex $V$,
then there exist subdomains $\Omega_{t_1}, \Omega_{t_2}, \dots, \Omega_{t_{N_{k,l}}}$ such that
$\alpha_k \leq \alpha_{t_1} \leq \alpha_{t_2} \leq \dots \leq \alpha_{t_{N_{k,l}}} \leq \alpha_l,$
and all the interfaces $\partial\Omega_k \cap \partial\Omega_{t_1}$, $\partial\Omega_{t_{N_{k,l}}} \cap \partial\Omega_l$, and $\partial\Omega_{t_i} \cap \partial\Omega_{t_{i+1}}$ for $i = 1, \dots, N_{k,l} - 1$ are faces containing $E$ or $V$.
\end{lemma}

\renewcommand{\proofname}{Proof}
\begin{proof}
Assume $\Gamma_{k,l}$ is an edge $E$. By the quasi-monotonicity condition, there exists $\Omega_{t_1} \in \Sigma_{p_1}$ with $n > p_1 > p$ such that $\partial\Omega_k \cap \partial\Omega_{t_1} = F_1$ is a face containing $E$.
When $\partial\Omega_{t_1} \cap \partial\Omega_l$ is also a face, the conclusion has been proved.
Otherwise, if $\partial\Omega_{t_1} \cap \partial\Omega_l = E$ is an edge, we can similarly find $\Omega_{t_2} \in \Sigma_{p_2}$ with $p_1 > p_2 > p$ such that $\partial\Omega_{t_1} \cap \partial\Omega_{t_2} = F_2$ is a face containing $E$, and so on. This process continues until some $\Omega_{t_{N_{k,l}}} \in \Sigma_{p_{N_{k,l}}}$ with $p_{N_{k,l}} > p$ satisfies that $\partial\Omega_{t_{N_{k,l}}} \cap \partial\Omega_l = F_{N_{k,l}+1}$ is a face containing $\Gamma_{k,l}$. By the definition of $\{\Sigma_l\}_{l=1}^m$, it follows that $\alpha_k \leq \alpha_{t_1} \leq \alpha_{t_2} \leq \dots \leq \alpha_{t_{N_{k,l}}} \leq \alpha_l.$
The same argument applies when $\Gamma_{k,l}$ is a vertex.
\end{proof}

For a subdomain $\Omega_k$, let ${\mathcal N}_{d,k}$ denote the set of nodes corresponding to $\{\mathcal T\}_d$ on $\Omega_k$. Let $S_d(\Omega_k)$ be the restriction of $S_d(\Omega)$ on $\Omega_k$.
For a given function $u \in H_0^1(\Omega)$, and let $u_{d,k}$ denote the $L^2$ projection of $u$ into $S_d(\Omega_k)$. Based on the multilayer model, we define the auxiliary function $u_d \in S_d(\Omega)$ as follows:

\begin{enumerate}
  \item[1)] For any \( \Omega_k \in \Sigma_1 \), define $u_h|_{\Omega_k} := u_{h,k}$.

\item[2)] For any \( \Omega_k \in \Sigma_2 \), define
\begin{equation*}
\label{eq:u_h 2}
u_d|_{\Omega_k}=
\begin{cases}
u_{d,l}, & \text{on } \Gamma_{k,l},\quad l\in\Lambda^1_k, \\
u_{d,k},& \text{elsewhere},
\end{cases}
\end{equation*}

  \item[3)] For any \( \Omega_k \in \Sigma_n \) with \( n \geq 3 \), define
  \[
  u_d|_{\Omega_k}=
  \begin{cases}
  u_{d,l}, & \text{on } \Gamma_{k,l}, \quad l \in \Lambda_k^1, \\
  u_{d,l}, & \text{on } \Gamma_{k,l} \setminus \bigcup_{1 \leq p \leq j-1} \bigcup_{t \in \Lambda_k^p} \Gamma_{k,t}, \quad l \in \Lambda_k^j,\ j = 2, \dots, n-1, \\
  u_{d,k}, & \text{elsewhere}.
  \end{cases}
  \]
\end{enumerate}

The following lemma provides a key estimate for the auxiliary function $u_h$, showing that $\|u-u_d\|_{L^2(\Omega_k)}^2$ is bounded by the seminorms of $u$ over $\Omega_k$ and adjacent subdomains with
coefficients no less than $\alpha_k$.

\begin{lemma}
\label{lem:the case satisfying Quasi-monotonicity Assumption}
Let $u \in H_0^1(\Omega)$, and suppose that the decomposition \eqref{eq:domain decomposition} satisfies the quasi-monotonicity condition. Let $u_d$ be the function defined above.
Then, for any $\Omega_k \in \Sigma_n$ with $n = 2, \dots, m$, the following estimate holds:
\begin{equation}
\label{eq:lem of the case1}
\alpha_k\|u-u_d\|_{L^2(\Omega_k)}^2
\lesssim \alpha_k d^2|u|_{H^1(\Omega_k)}^2+\sum_{j=1}^{n-1}\sum_{l\in\Lambda_k^{j}}\alpha_l d^2|u|_{H^1(\Omega_l)}^2,
\end{equation}
where the hidden constant depends only on the distribution of the coefficients $\{\alpha_i\}_{i=1}^{N_0}$ (may depend on the layer number $m$),
but is independent of the variations of the coefficients and the mesh size $d$.
\end{lemma}

\begin{proof} We follow the ideas in the proof of Lemma 5.4 in  \cite{Hu2023}.
We prove the lemma by considering the cases $n = 1$, $n = 2$ and $n \geq 3$, respectively. For $\Omega_k \in \Sigma_1$, by the property of the $L^2$ projection, we have
\begin{equation*}
    \alpha_k\|u -u_d\|_{L^2(\Omega_k)}^2 \lesssim \alpha_k d^2 |u|_{H^1(\Omega_k)}^2.
\end{equation*}

For $\Omega_k \in \Sigma_2$, applying the discrete $L^2$ norm yields
\[
\alpha_k \|u_{d,k}-u_d\|_{L^2(\Omega_k)}^2
\lesssim \alpha_k d^3 \sum_l \sum_{x_i \in \mathcal{N}_d(\Gamma_{k,l})} (u_{d,k} - u_{d,l})(x_i)^2.
\]
By Lemma~\ref{lem:sigma_lsigma_l+1}, each $\Gamma_{k,l}$ consists only of faces. Following the argument of Lemma 4.6 in \cite{bramble1991some}, and using the discrete $L^2$ norm, the $\varepsilon$-inequality, and the projection property, we obtain
\[
\begin{aligned}
    d^3 \sum_l \sum_{x_i \in \mathcal{N}_d(\Gamma_{k,l})} (u_{d,k} - u_{d,l})(x_i)^2
    &\lesssim d \sum_l \|u_{d,k} - u_{d,l}\|_{L^2(\Gamma_{k,l})}^2 \\
    &\lesssim d \|u_{d,k} - u\|_{L^2(\partial\Omega_k)}^2 + d \sum_l \|u_{d,l} - u\|_{L^2(\partial\Omega_l)}^2 \\
    &\lesssim d^2 |u|_{H^1(\Omega_k)}^2 + \sum_l d^2 |u|_{H^1(\Omega_l)}^2.
\end{aligned}
\]
Combining these estimates with the projection property gives
\begin{equation*}
\begin{aligned}
\alpha_k \|u -u_d\|_{L^2(\Omega_k)}^2
&\lesssim \alpha_k \|u - u_{d,k}\|_{L^2(\Omega_k)}^2 + \alpha_k \|u_{d,k} -U_d\|_{L^2(\Omega_k)}^2 \\
&\lesssim \alpha_k d^2 |u|_{H^1(\Omega_k)}^2 + \sum_{l \in \Lambda_k^1} \alpha_l d^2 |u|_{H^1(\Omega_l)}^2,
\end{aligned}
\end{equation*}
where $\Lambda_k^1 = \{ l : \Omega_l \cap \Omega_k \neq \emptyset,\ \Omega_l \in \Sigma_1 \}$.

For $\Omega_k \in \Sigma_n$ with $n \geq 3$, we estimate $I_1 := d^3 \sum_{x_i \in \mathcal{N}_d(\Gamma_{k,l})} (u_{d,k} - u_{d,l})(x_i)^2$. When $\Omega_l \in \Sigma_{n-1}$ and $u_d = u_{d,l}$ on $\Gamma_{k,l}$, the estimate follows the same argument as in the case $\Omega_k \in \Sigma_2$. When $\Omega_l \in \Sigma_p$ for $p \leq n - 2$, the interface $\Gamma_{k,l}$ may consist of faces, isolated edges, or isolated vertices. Accordingly, we partition the index set $\Lambda_k^p$ as follows:
\[
\Lambda_k^p = \Lambda_k^{p,F} \cup \Lambda_k^{p,E} \cup \Lambda_k^{p,V} \cup \Lambda_k^{p,\text{other}}, \quad p = 1,2,\dots,n-2,
\]
where
\[
\begin{aligned}
&\Lambda_k^{p,F} = \{ l : \Omega_l \cap \Omega_k \text{ is a face on which }u_d = u_{d,l},\ \Omega_l \in \Sigma_p \}, \\
&\Lambda_k^{p,E} = \{ l : \Omega_l \cap \Omega_k \text{ is an edge on which }u_d = u_{d,l},\ \Omega_l \in \Sigma_p \}, \\
&\Lambda_k^{p,V} = \{ l : \Omega_l \cap \Omega_k \text{ is a vertex on which }u_d = u_{d,l},\ \Omega_l \in \Sigma_p \}, \\
&\Lambda_k^{p,\text{other}} = \Lambda_k^p \setminus (\Lambda_k^{p,F} \cup \Lambda_k^{p,E} \cup \Lambda_k^{p,V}).
\end{aligned}
\]

For $l \in \Lambda_k^{p,F}$, the estimate follows as before. For $l \in \Lambda_k^{p,E}$, by Lemma~\ref{lem:sigma_lsigma_p p<l-1}, there exists a sequence $\Omega_{t_1}, \dots, \Omega_{t_{N_{k,l}}}$ such that
$\alpha_k \leq \alpha_{t_1} \leq \dots \leq \alpha_{t_{N_{k,l}}} \leq \alpha_l$,
and each pair intersects on a face containing $\Gamma_{k,l}$. Hence,
\begin{equation}
\label{eq:gamma_kl is a edge}
\begin{aligned}
 I_1&\lesssim d^3 \sum_{x_i \in \mathcal{N}_d(\Gamma_{k,l})} (u_{d,k} - u_{h,t_1})(x_i)^2 +  d^3 \sum_{x_i \in \mathcal{N}_d(\Gamma_{k,l})} (u_{d,t_{N_{k,l}}} - u_{d,l})(x_i)^2\\
&+ d^3 \sum_{x_i \in \mathcal{N}_d(\Gamma_{k,l})} \sum_{i=1}^{N_{k,l}-1} (u_{d,t_i} - u_{d,t_{i+1}})(x_i)^2.
\end{aligned}
\end{equation}
Using the same tools as before, we derive
\begin{equation*}
\label{eq:gamma_kl is a egde2}
\begin{aligned}
I_1&\lesssim d \|u - u_{d,l}\|_{L^2(\partial \Omega_k)}^2 + d \sum_{i=1}^{N_{k,l}} \|u - u_{d,t_i}\|_{L^2(\partial \Omega_{t_i})}^2  + h \|u - u_{d,l}\|_{L^2(\partial \Omega_l)}^2 \\
&\lesssim d^2 |u|_{H^1(\Omega_k)}^2 + d^2 \sum_{i=1}^{N_{k,l}} |u|_{H^1(\Omega_{t_i})}^2 + d^2 |u|_{H^1(\Omega_l)}^2.
\end{aligned}
\end{equation*}
For $l \in \Lambda_k^{p,V}$, a similar argument applies. Define $\Lambda_k^{p,EV} := \Lambda_k^{p,E} \cup \Lambda_k^{p,V}$. Combining all cases and the projection properties, we obtain for $\Omega_k \in \Sigma_n$ with $n \geq 3$,
\begin{equation*}
\label{eq:合并}
\begin{aligned}
\|u -u_d\|_{L^2(\Omega_k)}^2
&\lesssim d^2 \Big(|u|_{H^1(\Omega_k)}^2 + \sum_{j=1}^{n-1} \sum_{l \in \Lambda_k^j} |u|_{H^1(\Omega_l)}^2 + \sum_{p=1}^{n-2} \sum_{l \in \Lambda_k^{p,EV}} \sum_{i=1}^{N_{k,l}} |u|_{H^1(\Omega_{t_i})}^2\Big).
\end{aligned}
\end{equation*}
Moreover, since $\alpha_k \leq \alpha_l$ for all $l \in \Lambda_k^j$, we have
\begin{equation*}
\alpha_k \|u -u_d\|_{L^2(\Omega_k)}^2
\lesssim \alpha_k d^2 |u|_{H^1(\Omega_k)}^2 + d^2 \sum_{j=1}^{n-1} \sum_{l \in \Lambda_k^j} \alpha_l |u|_{H^1(\Omega_l)}^2,
\end{equation*}
where $\Lambda_k^j = \{ l : \Omega_l \cap \Omega_k \neq \emptyset,\ \Omega_l \in \Sigma_j \}$.
Notice that the numbers of subdomains in $\Lambda_k^{p,E}$ and $\Lambda_k^{p,V}$ are finite, and each $\Omega_{t_i}$ appears at most $m$ times.
Then the hidden constant is independent of the specific values of $\{\alpha_i\}$ and $d$ (but may depend on the layer $m$).
\end{proof}

\begin{remark}
The index set $\Lambda_k^{p,\text{other}}$ does not appear explicitly in the proof, but subdomains indexed by it may still be selected as some $\Omega_{t_i}$.
\end{remark}

\renewcommand{\proofname}{Proof of Theorem~\ref{thm: quasi-monotonicity}}
\begin{proof}
By Lemma~\ref{lem:the case satisfying Quasi-monotonicity Assumption}, we obtain
\begin{equation}
\label{eq:th1}
\begin{aligned}
     \|u -u_d\|_{L^2_\alpha(\Omega)}^2
     \lesssim \sum_{n=1}^m \sum_{\Omega_k \in \Sigma_n} \alpha_k d^2 |u|_{H^1(\Omega_k)}^2
     + \sum_{n=2}^m \sum_{\Omega_k \in \Sigma_n} d^2 \sum_{j=1}^{n-1} \sum_{l \in \Lambda_k^j} \alpha_l |u|_{H^1(\Omega_l)}^2.
\end{aligned}
\end{equation}

Since for each $j = 1, 2, \dots, n-1$, the set $\Lambda_k^j$ contains only those subdomains in $\Sigma_j$ that intersect with $\Omega_k$, and any given subdomain $\Omega_l$ can appear at most once in the summation $\sum_{j=1}^{n-1} \sum_{l \in \Lambda_k^j} \alpha_l |u|_{H^1(\Omega_l)}^2$, it follows that each term $d^2 \alpha_l |u|_{H^1(\Omega_l)}^2$ on the right-hand side of \eqref{eq:th1} appears no more than $m^2$ times. Combining the above estimate with the definition of the weighted $L^2$ projection, we obtain
\[
\| (I - Q_d^\alpha) u \|_{L^2_\alpha(\Omega)}^2 \leq \|u -u_d\|_{L^2_\alpha(\Omega)}^2 \leq C(m) d^2 |u|_{H^1_\alpha(\Omega)}^2,
\]
where $C(m)$ is a constant depending only on the distribution of $\{\alpha_k\}_{k=1}^{N_0}$, but independent of the specific values of $\alpha_k$ and the mesh size $h$.
\end{proof}

We proceed to prove Theorem~\ref{thm:the case satisfying no thorny vertex on S_h(Omega) refine}, beginning with a key lemma for the auxiliary function $u_d$. Its proof closely follows that of Lemma~\ref{lem:the case satisfying Quasi-monotonicity Assumption}, with modifications only in the treatment of thorny edges.

\begin{lemma}
\label{lem:the case satisfying no thorny vertex on S_h(Omega) refine}
Let $u \in S_{h}(\Omega)$, and assume that the decomposition \eqref{eq:domain decomposition} contains no thorny vertices. Let $u_d$ be defined as in the paragraph following \eqref{eq:lambda_k^j}. Then for any $\Omega_k \in \Sigma_n$ with $n = 2, \dots, m$, we have
\begin{equation*}
\alpha_k\|u -u_d\|_{L^2(\Omega_k)}^2
\lesssim d^2\log(H/d)\log(H/h) \big( \alpha_k |u|_{H^1(\Omega_k)}^2 + \sum_{j=1}^{n-1} \sum_{l \in \Lambda_k^j} \alpha_l |u|_{H^1(\Omega_l)}^2 \big),
\end{equation*}
where the hidden constant depends only on the distribution of $\{\alpha_i\}_{i=1}^{N_0}$ (may depend on the layer number $m$), but is independent of the variants of the coefficients, the mesh sizes $d$ and $h$.
\end{lemma}

\renewcommand{\proofname}{Proof}
\begin{proof}
We consider three cases: $n = 1$, $n = 2$ and $n \geq 3$. The case $\Omega_k \in \Sigma_1$ follows directly from the $L^2$ projection property, as in Lemma~\ref{lem:the case satisfying Quasi-monotonicity Assumption}. We also decompose each index set $\Lambda_k^p$ into four parts: $\Lambda_k^{p,F}$, $\Lambda_k^{p,E}$, $\Lambda_k^{p,V}$ and $\Lambda_k^{p,\text{other}}$.

For each $\Omega_k \in \Sigma_2$, the sets $\Lambda_k^{p,V}$ and $\Lambda_k^{p,\text{other}}$ are empty. Applying the discrete $L^2$ norm, we expand
\[
\alpha_k \|u -u_d\|_{L^2(\Omega_k)}^2
\lesssim \alpha_k \|u - u_{d,k}\|_{L^2(\Omega_k)}^2
+ \alpha_k d^3 \sum_l \sum_{x_i \in \mathcal{N}_d(\Gamma_{k,l})} (u_{d,k} - u_{d,l})(x_i)^2.
\]
We next estimate the second term. When $\Gamma_{k,l}$ is a face, the bound follows directly from Lemma~\ref{lem:the case satisfying Quasi-monotonicity Assumption}. We now consider the case with $\Gamma_{k,l}$ being an edge. The discrete $L^2$ norm gives
\begin{equation}
\label{eq:lem s_h 1}
\begin{aligned}
I_1 &:= d^3 \sum_{l} \sum_{x_i \in \mathcal{N}_d(\Gamma_{k,l})} (u_{d,k} - u_{d,l})(x_i)^2
\lesssim d^2 \sum_{l \in \Lambda_k^{1,E}} \|u_{d,k} - u_{d,l}\|_{L^2(\Gamma_{k,l})}^2 \\
&\lesssim d^2 \sum_{l \in \Lambda_k^{1,E}}
\big( \|u_{d,k} - \ga_{\Gamma_{k,l}}(u)\|_{L^2(\Gamma_{k,l})}^2
+ \|u_{d,l} - \ga_{\Gamma_{k,l}}(u)\|_{L^2(\Gamma_{k,l})}^2 \big),
\end{aligned}
\end{equation}
where $\ga_{\Gamma_{k,l}}(u)$ denotes the average of $u$ over $\Gamma_{k,l}$. By the edge Lemma~\ref{lem:edge lemma} (replacing $h$ by $d$), we further obtain
\[
I_1 \lesssim d^2 \log (H/d) \sum_{l \in \Lambda_k^{1,E}}
\big( \|u_{d,k} - \ga_{\Gamma_{k,l}}(u)\|_{H^1(\Omega_k)}^2
+ \|u_{d,l} - \ga_{\Gamma_{k,l}}(u)\|_{H^1(\Omega_l)}^2 \big).
\]

We now estimate $\|u_{d,k} - \ga_{\Gamma_{k,l}}(u)\|_{H^1(\Omega_k)}^2$. Let $r_{\partial \Omega_k}(u)$ be the average of $u$ over $\partial \Omega_k$. Then,
\begin{equation}
\label{eq:lem s_h 2}
\|u_{d,k} - \ga_{\Gamma_{k,l}}(u)\|_{H^1(\Omega_k)}^2
\lesssim \|u_{d,k} -\ga_{\partial \Omega_k}(u)\|_{H^1(\Omega_k)}^2
+ \|\ga_{\Gamma_{k,l}}(u -\ga_{\partial \Omega_k}(u))\|_{H^1(\Omega_k)}^2.
\end{equation}

For the first term on the right-hand side of \eqref{eq:lem s_h 2}, by the $L^2$ projection property and the Poincaré inequality, we have
\[
\|u_{d,k} -\ga_{\partial \Omega_k}(u)\|_{H^1(\Omega_k)}^2
\lesssim \|u - u_{d,k}\|_{H^1(\Omega_k)}^2
+ \|u -\ga_{\partial \Omega_k}(u)\|_{H^1(\Omega_k)}^2
\lesssim |u|_{H^1(\Omega_k)}^2.
\]

For the second term, since $\ga_{\Gamma_{k,l}}(u -\ga_{\partial \Omega_k}(u))$ is constant, its $H^1$-norm reduces to its $L^2$-norm. Noting that both $\Omega_k$ and $\Gamma_{k,l}$ are of size $O(H)$, we apply the Cauchy-Schwarz inequality, the edge Lemma~\ref{lem:edge lemma} and the Poincaré inequality to obtain
\[
\|\ga_{\Gamma_{k,l}}(u -\ga_{\partial \Omega_k}(u))\|_{H^1(\Omega_k)}^2
\lesssim \|u -\ga_{\partial \Omega_k}(u)\|_{L^2(\Gamma_{k,l})}^2
\lesssim \log(H/h)\, |u|_{H^1(\Omega_k)}^2.
\]
Thus,
\[
\|u_{d,k} - \ga_{\Gamma_{k,l}}(u)\|_{H^1(\Omega_k)}^2
\lesssim \log(H/h)\, |u|_{H^1(\Omega_k)}^2,
\]
and a similar estimate holds for $\|u_{d,l} - \ga_{\Gamma_{k,l}}(u)\|_{H^1(\Omega_l)}^2$. Combining all the above, we conclude that for $\Omega_k \in \Sigma_2$,
\[
\alpha_k \|u -u_d\|_{L^2(\Omega_k)}^2
\lesssim d^2 \log(H/d)\log(H/h)\big( \alpha_k |u|_{H^1(\Omega_k)}^2
+ \sum_{l \in \Lambda_k^1} \alpha_l |u|_{H^1(\Omega_l)}^2 \big).
\]

For each $\Omega_k \in \Sigma_n$ with $n \geq 3$, the analysis follows that of Lemma~\ref{lem:the case satisfying Quasi-monotonicity Assumption}. For $l \in \Lambda_k^{p,F}$, the estimate follows from the proof of Lemma~\ref{lem:the case satisfying Quasi-monotonicity Assumption}; for $l \in \Lambda_k^{p,E}$, the argument for the case $\Omega_k \in \Sigma_2$ applies. For $l \in \Lambda_k^{p,V}$, let $\textbf{x} := \Gamma_{k,l}$. By Lemma~\ref{lem:sigma_lsigma_p p<l-1}, there exist subdomains $\Omega_{t_1}, \dots, \Omega_{t_{N_{k,l}}}$ such that
$\alpha_k \leq \alpha_{t_1} \leq \dots \leq \alpha_{t_{N_{k,l}}} \leq \alpha_l$,
with each pair intersecting on a face or edge containing $\Gamma_{k,l}$. If the intersection is a face, we apply the same argument as in Lemma~\ref{lem:the case satisfying Quasi-monotonicity Assumption}; if it is an edge, we apply the argument used in the $\Sigma_2$ case. In either situation, we obtain
\begin{equation*}
d^3 (u_{d,k} - u_{d,l})(x)^2
\lesssim d^2 \log(H/d)\log(H/h)\big( |u|_{H^1(\Omega_k)}^2 + \sum_{i=1}^{N_{k,l}} |u|_{H^1(\Omega_{t_i})}^2 + |u|_{H^1(\Omega_l)}^2 \big).
\end{equation*}

Combining all three cases and using the fact that $\alpha_k \leq \alpha_l$ for all $l \in \Lambda_k^j$, $j = 1, \dots, n-1$, we obtain the desired estimate for $\Omega_k \in \Sigma_n$ with $n =2,\dots,m$:
\begin{equation*}
\alpha_k \|u -u_d\|_{L^2(\Omega_k)}^2
\lesssim d^2 \log(H/d)\log(H/h)\big( \alpha_k |u|_{H^1(\Omega_k)}^2
+ \sum_{j=1}^{n-1} \sum_{l \in \Lambda_k^j} \alpha_l |u|_{H^1(\Omega_l)}^2 \big),
\end{equation*}
where the hidden constant depends only on the distribution of $\{\alpha_i\}$, but is independent of the variants of the coefficients and the mesh size $h$.
\end{proof}

\renewcommand{\proofname}{Proof of Theorem~\ref{thm:the case satisfying no thorny vertex on S_h(Omega) refine}}
\begin{proof}
The proof follows that of Theorem~\ref{thm: quasi-monotonicity}, with Lemma~\ref{lem:the case satisfying no thorny vertex on S_h(Omega) refine} replacing Lemma~\ref{lem:the case satisfying Quasi-monotonicity Assumption}.
\end{proof}

\subsection{Proof of the general case}
This section is devoted to proofs of Theorem~\ref{thm:thorny vertex and edge} and Theorem~\ref{thm:thorny vertex and edge refine}. We first introduce auxiliary subdomains related to thorny vertices and thorny edges, following Section~5.2.1 of \cite{Hu2023}. These subdomains partition $\Omega$ into two parts: one containing all thorny features, and the other satisfying the quasi-monotonicity condition. The former is treated separately, while the latter is estimated using Theorem~\ref{thm: quasi-monotonicity}.

For each thorny edge $E_j$, let $\mathscr{S}_{E_j}^* = \{\Omega_{j,t}\}_{t=1}^{m_j}$ as defined at the beginning of Subsection~\ref{sec:main results for error estimates}. For each $\Omega_{j,t}$, define an auxiliary subdomain $D_{j,t} \subset \Omega_{j,t}$ containing $E_j$, satisfying:
1) $D_{j,t}$ is the union of some ${\mathcal T}_d$-elements in $\Omega_{j,t}$;
2) $D_{j,t}$ is a Lipschitz domain of size $O(H)$, with Lipschitz constant independent of the mesh size $d$ (and $h$);
3) for any distinct thorny features $E_i$ and $E_j$ (or $V_j$), their associated auxiliary subdomains intersect at most at a single vertex, i.e., $E_i \cap E_j$ (or $E_i \cap V_j$).

Likewise, for each thorny vertex $V_i$ with $\mathscr{S}_{V_i}^* = \{\Omega_{i,r}\}_{r=1}^{m_i}$, define $D_{i,r} \subset \Omega_{i,r}$ containing $V_i$, satisfying conditions 1) and 2). Moreover, the auxiliary subdomains corresponding to different thorny vertices  are disjoint.

\begin{figure}[!htbp]
    \centering
    \includegraphics[width=0.8\textwidth]{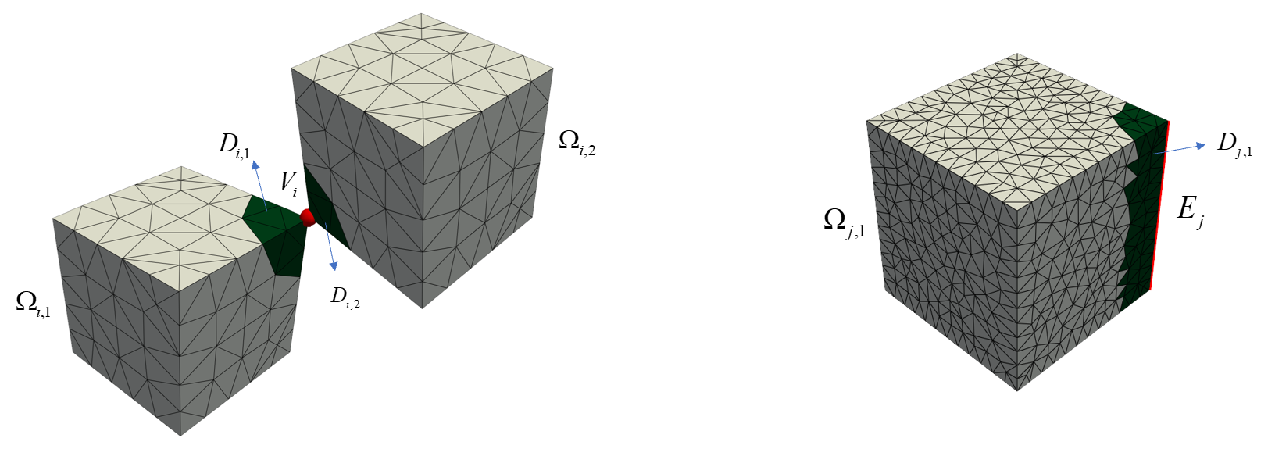}
    \caption{Illustration of auxiliary subdomains. Left: auxiliary subdomains \( D_{i,1} \) and \( D_{i,2} \) associated with a thorny vertex \( V_i \); Right: auxiliary subdomain \( D_{j,1} \) associated with a thorny edge \( E_j \).}
    \label{fig:axu subdomain}
\end{figure}

Following Section~5.2.1 of \cite{Hu2023}, we define several sets associated with the auxiliary subdomains:
\[
D_{V_i}^* = \bigcup_{r=1}^{m_i} D_{i,r}, \quad \Omega_{V_i}^* = \bigcup_{r=1}^{m_i} \Omega_{i,r}, \quad
D_{E_j}^* = \bigcup_{t=1}^{n_j} D_{j,t}, \quad \Omega_{E_j}^* = \bigcup_{t=1}^{n_j} \Omega_{j,t}.
\]
\[
\Omega_k^\partial = \Omega_k \setminus \bigcup_{D_{i,r} \subset \Omega_k} D_{i,r}~(\Omega_k \in \mathscr{S}^*), \qquad D^\partial = \bigcup_{\Omega_k \in \mathscr{S}^*} \Omega_k^\partial,
\]
where each $\Omega_k^\partial$ is a Lipschitz domain with Lipschitz constant independent of $d$ (and $h$). Define
\begin{equation}
\label{eq:Omega^c}
\Omega^c = \Big( \bigcup_{V_i \in \mathcal{V}^*} \bigcup_{\Omega_k \in \mathscr{S}_{V_i}^c} \Omega_k \Big) \cup \Big( \bigcup_{E_j \in \mathcal{E}^*} \bigcup_{\Omega_k \in \mathscr{S}_{E_j}^c} \Omega_k \Big),
\end{equation}
\begin{equation}
\label{eq:D^*}
D^* = \Big( \bigcup_{V_i \in \mathcal{V}^*} D_{V_i}^* \Big) \cup \Big( \bigcup_{E_j \in \mathcal{E}^*} D_{E_j}^* \Big),\quad \Omega_*^c = \bigcup_{\Omega_k \in \mathscr{S}_*^c} \Omega_k, \quad \Omega^\partial = D^\partial \cup \Omega_*^c.
\end{equation}
Consequently, the entire domain $\Omega$ is decomposed as
\begin{equation}
\label{eq:Omega=D^*cupOmega^partial}
\Omega = D^* \cup D^\partial \cup \Omega_*^c = D^* \cup \Omega^\partial.
\end{equation}

We now derive estimates on the region $D^*$ containing all thorny vertices and thorny edges. To this end, we define an auxiliary function $u_d$ on $D^*$ and state a key lemma. Let $u_{d,ir}$ be the $L^2$ projection of $u$ onto $\Omega_{i,r}$, and $P_F u$ the $L^2$ projection onto a face $F$. For each thorny vertex $V_i$ (or edge $E_i$), define $u_d$ on $D_{i,r} \subset \Omega_{i,r}$ by
\begin{equation}
\label{eq:U_d D^*}
u_d|_{D_{i,r}} =
\begin{cases}
P_F u, & \text{at interior nodes of faces } F \subset \partial D_{i,r}, \\
0, & \text{at all nodes on edges } E \subset \partial D_{i,r}, \\
u_{d,ir}, & \text{at all other nodes}.
\end{cases}
\end{equation}

The proof of the following lemma is similar in the spirit to Lemma 4.6 in \cite{bramble1991some}, but differs in that the subdomains here are general Lipschitz domains, and the edges and faces may be irregular.
Thus, the edge lemma established in Section~\ref{sec:Edge and Face Lemmas on Lipschitz Domains} needs to be used.

\begin{lemma}
\label{lem:estimate on D^*} Let $\mathscr{S}^*$, $D^*$ and $u_d$ be defined as in \eqref{eq:s*}, \eqref{eq:D^*} and \eqref{eq:U_d D^*}, respectively. Then, for any $u \in H_0^1(\Omega)$, we have
\begin{equation*}
\| u -u_d \|^2_{L^2_\alpha(D^*)} \leq C d^2\log(H/d) \sum_{\Omega_i \in \mathscr{S}^*} \alpha_i \| u \|^2_{H^1(\Omega_i)},
\end{equation*}
where $C$ is a constant independent of $\{\alpha_k\}$ and the mesh size $d$.
\end{lemma}

\renewcommand{\proofname}{Proof}
\begin{proof}
Fix a subdomain $D_{i,r} \subset \Omega_{i,r}$. By the triangle inequality and the discrete $L^2$ norm, yields
\begin{equation}
\label{eq:estimate on D^*1}
\begin{aligned}
\| u -u_d \|^2_{L^2(D_{i,r})} \lesssim \| u - u_{d,ir} \|^2_{L^2(D_{i,r})} + d^3 \sum_{x_i \in \mathcal{N}_d(D_{i,r})} (u_{d,ir} -u_d)(x_i)^2.
\end{aligned}
\end{equation}
From the definition of $u_d|_{D_{i,r}}$, we further estimate
\begin{equation*}
\begin{aligned}
&d^3 \sum_{x_i \in \mathcal{N}_d(D_{i,r})} (u_{d,ir} -u_d)(x_i)^2\\
&\lesssim d^3 \sum_F \sum_{x_i \in \mathcal{N}_d(F)} (u_{d,ir} - P_F u)(x_i)^2
+ d^3 \sum_E \sum_{x_i \in \mathcal{N}_d(E)} u_{d,ir}(x_i)^2 \\
&:= I_1 + I_2.
\end{aligned}
\end{equation*}
where \( F \subset \partial D_{i,r} \) and \( E \subset \partial D_{i,r} \) denote a face or an edge of \( D_{i,r} \), respectively.

For the first term $I_1$, using the discrete $L^2$ norm, the $\varepsilon$-inequality, and properties of the $L^2$ projection, we obtain
\begin{equation}
\label{eq:estimate on D^*2}
\begin{aligned}
I_1 &\lesssim d \sum_F \|u_{d,ir} - P_F u\|^2_{L^2(F)}
\lesssim d \|u - u_{d,ir}\|^2_{L^2(\partial D_{i,r})} \\
&\lesssim \|u - u_{d,ir}\|^2_{L^2(D_{i,r})} + d^2 \|u - u_{d,ir}\|^2_{H^1(D_{i,r})}
\lesssim d^2 |u|^2_{H^1(\Omega_{i,r})}.
\end{aligned}
\end{equation}

For the second term $I_2$, using the discrete $L^2$ norm, Lemma~\ref{lem:edge lemma} (replacing $h$ by $d$) and the projection property, we get
\begin{equation}
\label{eq:estimate on D^*3}
I_2 \lesssim d^2 \sum_E \|u_{d,ir}\|^2_{L^2(E)}
\lesssim d^2 \log(H/d) \|u\|^2_{H^1(\Omega_{i,r})}.
\end{equation}

Substituting \eqref{eq:estimate on D^*2} and \eqref{eq:estimate on D^*3} into \eqref{eq:estimate on D^*1}, and applying the projection property again, we obtain
\[
\| u -u_d \|^2_{L^2(D_{i,r})} \lesssim d^2 \log(H/d)\|u\|^2_{H^1(\Omega_{i,r})}.
\]

Since the number of thorny vertices and edges is finite and depends only on the distribution of $\{\alpha_k\}_{k=1}^{N_0}$, we conclude
\begin{equation*}
\| u -u_d \|^2_{L^2_\alpha(D^*)}
\lesssim d^2 \log(H/d) \sum_{D_{i,r}} \alpha_{i,r} \|u\|^2_{H^1(\Omega_{i,r})}
\lesssim d^2\log(H/d) \sum_{\Omega_i \subset \mathscr{S}^*} \alpha_i \|u\|^2_{H^1(\Omega_i)},
\end{equation*}
\end{proof}


We next estimate the error on subdomains influenced by auxiliary subdomains.  Let \( G \subset \Omega^c \cup D^{\partial} \) be a subdomain intersecting \( D^* \) along its boundary, and define
\[
\Gamma_G := \partial G \cap \partial D^*.
\]
Assume \( D_{i,r} \cap G \neq \emptyset \). On the interface \( \Gamma_G \cap \partial D_{i,r} \), the auxiliary function \(U_d \) is defined according to its values in \( D_{i,r} \):
\begin{equation}
\label{eq:U_d G}
u_d|_{\Gamma_G \cap \partial D_{i,r}} =
\begin{cases}
P_F u, & \text{at interior nodes on faces } F \subset (\Gamma_G \cap \partial D_{i,r}), \\
0, & \text{at all nodes on edges } E \subset (\Gamma_G \cap \partial D_{i,r}).
\end{cases}
\end{equation}

The following lemma handles potential isolated edges or vertices in $\Gamma_G$.
\begin{lemma}
\label{lem:Gamma_G-estimate}
Let \( G \subset D^{\partial} \cup \Omega^c \) intersect \( D_{V_i}^* \) (or \( D_{E_i}^* \)) at an isolated edge \( E \) or vertex \( V \). Then there exists a subdomain \( D_{i,r} \subset D_{V_i}^* \) (or \( D_{E_i}^* \)) and a sequence of subdomains \( \{\Omega_{t_i}\}_{i=1}^{N_{G,ir}} \) such that \( E \subset \partial D_{i,r} \), and the coefficient values satisfy \( \alpha_G \leq \alpha_{t_1} \leq \cdots \leq \alpha_{t_{N_{G,ir}}} \leq \alpha_{i,r} \). Moreover, each of the intersections \( G \cap \Omega_{t_1} \), \( \Omega_{t_{N_{G,ir}}} \cap \Omega_{i,r} \), and \( \Omega_{t_i} \cap \Omega_{t_{i+1}} \) (for \( i = 1, \dots, N_{G,ir}-1 \)) is a face containing \( E \) or \( V \).
\end{lemma}
\begin{proof}
Suppose \( G \subset D^\partial \cup \Omega^c \) intersects \( D_{V_i}^* \) (or \( D_{E_i}^* \)) at an isolated edge \( E \). Since no auxiliary subdomain is constructed near \( E \), there exists \( \Omega_{t_1} \) such that \( \Omega_{t_1} \cap G \supset E \) is a face and \( \alpha_{t_1} \geq \alpha_G \). If \( E \) is a thorny edge of \( \Omega_{t_1} \), or contains a thorny vertex of \( \Omega_{t_1} \), then there exists some \( D_{i,r} \subset \Omega_{t_1} \) such that \( E \subset D_{i,r} \).
 Otherwise, we iteratively construct a sequence \( \{\Omega_{t_i}\} \) such that each intersection \( \Omega_{t_i} \cap \Omega_{t_{i-1}} \supset E \) is a face and \( \alpha_{t_i} \geq \alpha_{t_{i-1}} \), continuing this process until \( E \subset D_{i,r} \subset \Omega_{i,r} \). The case of an isolated vertex \( V \) can be treated in a similar manner.
\end{proof}

Let \( u_{d,G} \) denote the \( L^2 \) projection of \( u \) onto \( G \). By the triangle inequality and the discrete \( L^2 \) norm, we have
\begin{equation}
\label{eq:G-estimate1}
\| u -u_d \|^2_{L^2(G)} \lesssim \| u - u_{d,G} \|^2_{L^2(G)} + d^3 \sum_{x_i \in \mathcal{N}_d(G)} (u_{d,G} -u_d)(x_i)^2.
\end{equation}
\begin{equation}
\label{eq:G-estimate2}
d^3\!\!\! \sum_{x_i \in \mathcal{N}_d(G)}\!\! (u_{d,G} -u_d)(x_i)^2 = d^3 \Big(\!\sum_{x_i \in \mathcal{N}_d(\Gamma_G)}\!\!\! (u_{d,G} -u_d)(x_i)^2 +\!\!\!\sum_{x_i \in \mathcal{N}_d(G \backslash \Gamma_G)}\!\!\!(u_{d,G} -u_d)(x_i)^2\Big).
\end{equation}
Since the portion \( G \setminus \Gamma_G \) satisfies the quasi-monotonicity condition, we now estimate the contribution from \( \Gamma_G \). Define the set of auxiliary subdomains intersecting \( G \) by
\[
\Lambda_G^* = \{ D_{i,r} : \partial D_{i,r} \cap \partial G \neq \emptyset \}.
\]

\begin{lemma}
\label{lem:estimate on Gamma_G}
Let $u \in H_0^1(\Omega)$ and $G \subset D^{\partial} \cup \Omega^c$ be a subdomain. Let $u_{d,G}$ be the $L^2$ projection of $u$ onto $G$, and $U_d$ be defined as in \eqref{eq:U_d G}. Then,
\begin{equation}
\label{eq:Gamma_G estimate}
d^3\sum_{x_i\in \mathcal{N}_d(\Gamma_G)}(u_{d,G}-u_d )(x_i)^2 \lesssim d^2|u|^2_{H^1(G)} + \sum_{D_{i,r}\in\Lambda_G^*} d^2\log(H/d) \|u\|^2_{H^1(\Omega_{i,r})},
\end{equation}
where the constant $C$ depends on the distribution of $\{\alpha_k\}$ but is independent of the variants of the coefficients and the mesh size $d$.
\end{lemma}

\begin{proof}
We analyze the cases where $\Gamma_G$ contains isolated vertices, edges, or faces. Suppose $V \in \Gamma_G$ is an isolated point. By Lemma~\ref{lem:Gamma_G-estimate}, there exists $D_{i,r}$ such that $G \cap D_{i,r} = V$. By the definition of $u_d$, we have
\[
d^3(u_{d,G} -u_d)(V)^2\lesssim d^3 (u_{d,G} - u_{d,ir})(V)^2 + d^3 u_{d,ir}(V)^2.
\]
By Lemma~\ref{lem:Gamma_G-estimate}, the first term can be converted into estimates on faces containing $V$, as in the proof of Lemma~\ref{lem:the case satisfying Quasi-monotonicity Assumption}. Hence, it can be merged with the quasi-monotonic part $G \setminus \Gamma_G$. For the second term, the embedding theorem yields
\[
d^3 u_{d,ir}(V)^2 \lesssim d^3 \|u_{d,ir}\|^2_{L^\infty(\Omega_{i,r})} \lesssim d^2 \|u_{d,ir}\|^2_{H^1(\Omega_{i,r})} \lesssim d^2 \|u\|^2_{H^1(\Omega_{i,r})}.
\]

When \( G \) intersects  \( D_{i,r} \) only on an isolated edge \( E \), the analysis proceeds similarly to the vertex case. If \( G \) intersects \( D_{i,r} \) along a face \( F \), the argument follows that of Lemma~\ref{lem:estimate on D^*}. Combining all cases and noting that the number of thorny vertices and thorny edges is finite, we conclude that
\begin{align*}
d^3\sum_{x_i\in \mathcal{N}_d(\Gamma_G)}(u_{d,G}-U_d )(x_i)^2
&\lesssim d^2|u|^2_{H^1(G)} + \sum_{D_{i,r} \in \Lambda_G^*} d^2 \log(H/d) \|u\|^2_{H^1(\Omega_{i,r})}.
\end{align*}
\end{proof}

\begin{remark}
Terms such as \( d^3(u_{d,G} - u_{d,ir})(V)^2 \) are absorbed into the estimate over \( G \setminus \Gamma_G \) and hence are not explicitly included in the right-hand side of \eqref{eq:Gamma_G estimate}.
\end{remark}

Based on the above discussion, for the general case involving thorny vertices and edges, the auxiliary function \(u_d \in S_d(\Omega) \) is constructed in the following manner:
\begin{enumerate}
  \item[1)] On \( D^* \), define \(u_d \) as in~\eqref{eq:U_d D^*};
  \item[2)] On the interfaces \( \Gamma_G \), where \( G \subset D^{\partial} \cup \Omega^c \), define \(U_d \) according to~\eqref{eq:U_d G};
  \item[3)] In the remaining region \( \Omega^\partial \setminus \bigcup_{G} \Gamma_G \), define \(U_d \) following the description after~\eqref{eq:lambda_k^j}.
\end{enumerate}

\renewcommand{\proofname}{Proof of Theorem~\ref{thm:thorny vertex and edge}}
\begin{proof}
By \eqref{eq:Omega=D^*cupOmega^partial}, we have
\begin{equation*}
\label{eq:th2}
\begin{aligned}
\|u -u_d\|_{L^2_\alpha(\Omega)}^2 &= \|u -u_d\|_{L^2_\alpha(D^*)}^2 +  \!\!\!\!\sum_{G \subset D^\partial \cup \Omega^c} \!\!\!\!\|u -u_d\|_{L^2_\alpha(G)}^2 + \!\!\!\!\!\!\sum_{\Omega_i \subset \Omega^\partial \setminus (D^\partial \cup \Omega^c)}\!\! \!\!\!\!\|u -u_d\|_{L^2_\alpha(\Omega_i)}^2, \\
&:= I_1 + I_2 + I_3.
\end{aligned}
\end{equation*}
Applying Lemma~\ref{lem:estimate on Gamma_G}, we obtain
\begin{equation*}
 d^3 \!\!\!\sum_{G \subset D^\partial \cup \Omega^c} \sum_{x_i \in \mathcal{N}_d(\Gamma_G)} (u_{d,G} -u_d)(x_i)^2 \lesssim \!\sum_G d^2 |u|_{H^1(G)}^2 + \!\!\!\sum_{\Omega_{i,r} \in \mathscr{S}^*} \! d^2\log(H/d) \|u\|_{H^1(\Omega_{i,r})}^2.
\end{equation*}
Combining \eqref{eq:G-estimate1}, \eqref{eq:G-estimate2}, and {\bf Proposition 4.1}, we obtain
\begin{equation*}
\begin{aligned}
I_2 \lesssim \sum_G d^2 |u|_{H^1_\alpha(G)}^2 + \sum_{\Omega_{i,r} \in \mathscr{S}^*} \alpha_{i,r} d^2\log(H/d) \|u\|_{H^1(\Omega_{i,r})}^2 + d^3 \sum_{x_i \in G \setminus \Gamma_G} (u_{d,G} -u_d)(x_i)^2.
\end{aligned}
\end{equation*}
Since \( G \setminus \Gamma_G \) and \( \Omega_i \subset \Omega \setminus (D^\partial \cup \Omega^c) \) satisfy the quasi-monotonicity condition, Theorem~\ref{thm: quasi-monotonicity} yields
\[
d^3 \sum_{x_i \in \mathcal{N}_d(G \setminus \Gamma_G)} (u_{d,G} -u_d)(x_i)^2 + \sum_{\Omega_i \subset \Omega^\partial \setminus (D^\partial \cup \Omega^c)} \alpha_i \|u -u_d\|^2_{L^2(\Omega_i)} \lesssim d^2 |u|_{H^1_\alpha(\Omega)}^2.
\]
Combining the estimates above with Lemma~\ref{lem:estimate on D^*}, we conclude that
\[
\|(I - Q_d^\alpha) u\|_{L^2_\alpha(\Omega)}^2  \lesssim d^2\log(H/d)\sum_{\Omega_k \in \mathscr{S}^*} \alpha_k \|u\|_{H^1(\Omega_k)}^2 + d^2 \sum_{\Omega_k \notin \mathscr{S}^*} \alpha_k |u|_{H^1(\Omega_k)}^2.
\]
\end{proof}

We proceed to prove Theorem~\ref{thm:thorny vertex and edge refine}. Since the proof closely follows that of Theorem~\ref{thm:thorny vertex and edge}, we only highlight the differences and omit repeated arguments; for details, refer to the proof of Theorem~\ref{thm:thorny vertex and edge}.

We begin by adjusting the definitions of the sets related to auxiliary subdomains, excluding those associated with thorny edges. Define
\[
\tilde{\Omega}_k^\partial = \Omega_k \setminus \bigcup_{D_{ir} \subset \Omega_k} D_{ir} ~(\Omega_k \in \mathscr{\tilde{S}}^*), \quad \tilde{D}^\partial = \bigcup_{\Omega_k \in \mathscr{\tilde{S}}^*} \tilde{\Omega}_k^\partial,\quad\tilde{\Omega}^c = \bigcup_{V_i \in \mathcal{V}^*} \bigcup_{\Omega_k \in \mathscr{S}_{V_i}^c} \Omega_k,
\]
Similarly, define the following sets:
\begin{equation}
\label{eq:D^* refine}
\tilde{D}^* =  \bigcup_{V_i \in \mathcal{V}^*} D_{V_i}^*,\quad \tilde{\Omega}_*^c = \bigcup_{\Omega_k \in \mathscr{\tilde{S}}_*^c} \Omega_k, \quad \tilde{\Omega}^\partial =\tilde{D}^\partial \cup \tilde{\Omega}_*^c.
\end{equation}
Thus the entire domain \( \Omega \) can be decomposed as
\begin{equation}
\label{eq:Omega=D^*cupOmega^partial refine}
\Omega = \tilde{D}^* \cup \tilde{D}^\partial \cup \tilde{\Omega}_*^c = \tilde{D}^* \cup \tilde{\Omega}^\partial.
\end{equation}

Let the function \(u_d \) be defined as in \eqref{eq:U_d D^*}. Then, on the revised collections of auxiliary subdomains introduced above, we obtain the following two lemmas. Their proofs are analogous to those of Lemma~\ref{lem:estimate on D^*} and Lemma~\ref{lem:estimate on Gamma_G}.

\begin{lemma}
\label{lem:estimate on D^* refine}
Let \( \mathscr{\tilde{S}}^* \), \(u_d \) and \( \tilde{D}^* \) be defined by \eqref{eq:tilde s*}, \eqref{eq:U_d D^*} and \eqref{eq:D^* refine}, respectively. For any \( u \in S_{h}(\Omega) \),
we have
\begin{equation*}
\| u -u_d \|^2_{L^2_\alpha(\tilde{D}^*)} \leq C d^2\log(H/d) \sum_{\Omega_i \in \mathscr{\tilde{S}}^*} \alpha_i \| u \|^2_{H^1(\Omega_i)},
\end{equation*}
where \( C \) is a constant independent of \( \{ \alpha_k \} \), the mesh sizes $d$ and \( h \).
\end{lemma}

Let \( G \subset \tilde{D}^{\partial} \cup \tilde{\Omega}^c \) be a subdomain. Define
\[
\tilde{\Lambda}_G^* = \{ D_{i,r}\subset\tilde{D}^* : \partial D_{i,r} \cap \partial G \neq \emptyset \}, \quad \tilde{\Gamma}_G = \{ \partial \tilde{D}^* \cap \partial G : G \subset \tilde{D}^{\partial} \cup \tilde{\Omega}^c \}.
\]
According to \eqref{eq:U_d G}, we define the auxiliary function \(u_d \) on \( \tilde{\Gamma}_G \), and obtain the following lemma.

\begin{lemma}
\label{lem:estimate on Gamma_G refine}
Let \( u \in H_0^1(\Omega) \), and let \( G \subset \tilde{D}^{\partial} \cup \tilde{\Omega}^c \) be a subdomain. Let \( u_{d,G} \) denote the \( L^2 \) projection of \( u \) on \( G \), and let \(u_d \) be defined by \eqref{eq:U_d G}. Then,
\begin{equation}
\label{eq:Gamma_G的估计 refine}
d^3 \sum_{x_i \in \mathcal{N}_d(\tilde{\Gamma}_G)} (u_{d,G} -u_d)(x_i)^2 \lesssim d^2\log(H/d)\log(H/h)\big( |u|^2_{H^1(G)} + \sum_{D_{i,r} \in \tilde{\Lambda}_G^*} \| u \|^2_{H^1(\Omega_{i,r})} \big),
\end{equation}
where the hidden constant \( C \) depends on the distribution of \( \{ \alpha_k \} \), but is independent of the variants of the coefficients of \( \{ \alpha_k \} \), the mesh sizes $d$ and \( h \).
\end{lemma}

\renewcommand{\proofname}{Proof}
\begin{proof}
The proof of this lemma is similar to that of Lemma~\ref{lem:estimate on Gamma_G}. However, to handle edge-related terms in \( \tilde{\Gamma}_G \), such as
$d^3\sum_{x_i\in \mathcal{N}_d(E)}(u_{d,G}-u_{d,ir} )(x_i)^2$, one needs to adopt the technique used in the proof of Lemma~\ref{lem:the case satisfying no thorny vertex on S_h(Omega) refine}.
\end{proof}

\begin{remark}
As in Lemma~\ref{lem:estimate on Gamma_G}, certain terms in the conclusion of Lemma~\ref{lem:estimate on Gamma_G refine} are omitted, as they can be incorporated into the estimate over \( G \setminus \tilde{\Gamma}_G \).
\end{remark}

\renewcommand{\proofname}{Proof of Theorem~\ref{thm:thorny vertex and edge refine}}
\begin{proof}
The argument follows that of Theorem~\ref{thm:thorny vertex and edge}. We construct an auxiliary function \(u_d \in S_d(\Omega) \), modifying the procedure below Lemma~\ref{lem:estimate on Gamma_G}. In Step~(1), \( D^* \) is replaced by \( \tilde{D}^* \); the remaining steps are adjusted accordingly to complete the construction of \(u_d \).
The estimate then follows directly from Lemma~\ref{lem:estimate on D^* refine}, Lemma~\ref{lem:estimate on Gamma_G refine} and Theorem~\ref{thm:the case satisfying no thorny vertex on S_h(Omega) refine}.
\end{proof}

\renewcommand{\proofname}{Proof}

\bibliographystyle{siamplain}

\end{document}